\documentclass[12pt]{article}
\usepackage{amsmath}
\usepackage{amsfonts,amsmath,amssymb,amsthm}
\usepackage{graphicx,psfrag,epsf}
\usepackage{enumerate}
\usepackage[numbers]{natbib}
\RequirePackage[colorlinks,citecolor=blue,linkcolor=blue,urlcolor=blue]{hyperref}

\usepackage{url} 

\newcommand{\blind}{0}

\usepackage{bm}
\usepackage{tikz}
\usepackage[capitalise]{cleveref}

\usepackage[shortlabels]{enumitem}

\newtheorem{theorem}{Theorem}
\newtheorem{definition}{Definition}
\newtheorem{corollary}{Corollary}
\newtheorem{lemma}{Lemma}
\newtheorem{proposition}{Proposition}
\newtheorem{remark}{Remark}
\newcommand{\bw}{\mathbf{w}}
\newcommand{\bX}{\mathbf{X}}

\newcommand{\rb}{\mathrm{B}}

\newcommand{\iid}{\stackrel{iid}{\sim}}		      

\newcommand{\hth}{\hat{\theta}}
\newcommand{\hbh}{\hat{\beta}}

\newcommand{\tilth}{\tilde{\theta}}
\newcommand{\datmat}{\mathcal{D}_n}

\def\R{\mathbb{R}}
\def\E{\mathbb{E}}
\def\P{\mathbb{P}}
\def\cov{\mathrm{Cov}}
\def\var{\mathrm{Var}}
\def\V{\mathrm{Var}}

\def\one{\mathbf{1}}
\def\D{\mathcal{D}}

\def\U{\mathbf{u}}
\def\A{\mathrm{A}}
\def\B{\mathrm{B}}
\def\rd{\mathrm{D}}
\def\rp{\mathrm{P}}
\def\rg{\mathrm{G}}
\def\rh{\mathrm{H}}

\def\rl{\mathrm{L}}
\def\ru{\mathrm{U}}
\def\rv{\mathrm{V}}

\def\sij{\mathrm{ps}\text{-}\mathrm{IJ}}
\def\O\mathcal{O}
\def\hij{\widehat{\mathrm{IJ}}}
\def\ij{\mathrm{IJ}}

\newcommand{\bone}{\mathbf{1}}

\DeclareMathOperator*{\argmin}{arg\,min}

\begin{document}

\def\spacingset#1{\renewcommand{\baselinestretch}%
{#1}\small\normalsize} \spacingset{1}


\if0\blind
{
  \title{\bf Bias, Consistency, and Alternative Perspectives of the Infinitesimal Jackknife}
  \author{Wei Peng, Lucas Mentch\thanks{
    This work was partially supported by NSF DMS-1712041}  \hspace{.2cm}\\
    Department of Statistics, University of Pittsburgh\\
    and \\
    Leonard Stefanski\\
    Department of Statistics, North Carolina State University}
  \maketitle
} \fi

\if1\blind
{
  \bigskip
  \bigskip
  \bigskip
  \begin{center}
    {\LARGE\bf Infinitesimal Jackknife Perspectives}
\end{center}
  \medskip
} \fi

\bigskip
\begin{abstract}
	Though introduced nearly 50 years ago, the infinitesimal jackknife (IJ) remains a popular modern tool for quantifying predictive uncertainty in complex estimation settings.  In particular, when supervised learning ensembles are constructed via bootstrap samples, recent work demonstrated that the IJ estimate of variance is particularly convenient and useful. However, despite the algebraic simplicity of its final form, its derivation is rather complex.  As a result, studies clarifying the intuition behind the estimator or rigorously investigating its properties have been severely lacking. This work aims to take a step forward on both fronts. 
	We demonstrate that surprisingly, the exact form of the IJ estimator can be obtained via a straightforward linear regression of the individual bootstrap estimates on their respective weights or via the classical jackknife.  The latter realization is particularly useful as it allows us to formally investigate the bias of the IJ variance estimator and better characterize the settings in which its use is appropriate.
	Finally, we extend these results to the case of U-statistics where base models are constructed via subsampling rather than bootstrapping and provide a consistent estimate of the resulting variance.
	
\end{abstract}

\spacingset{1.5} 


\section{Introduction}
It is difficult to overstate the importance and utility of resampling methods and the bootstrap in particular for determining properties of estimators whenever exact, explicit sampling distributions cannot be readily determined.  Given a sample $X_1, ..., X_n \sim \mathrm{P}$, a parameter of interest $\theta$, and an estimator $\hat{\theta} = s(X_1, ..., X_n)$, it is often of interest to estimate, for example, $\V(\hat{\theta})$. Let $\bm{x}=(x_1, ..., x_n)$ denote the observed values of the sample so that for a particular realization, $\hat{\theta} = s(\bm{x})$.  To provide bootstrap estimate of the variability of the estimator, we can draw $B$ (re)samples of size $n$ with replacement from $\{x_1, ..., x_n\}$ to form bootstrap samples $\bm{x}^{*}_{1}, ..., \bm{x}^{*}_{B}$ from which we calculate bootstrap estimates $\hat{\theta}_1, ..., \hat{\theta}_B$.  The nonparametric bootstrap variance estimate of $\hth$ is then taken as the empirical variance of $\hth_1, ..., \hth_B$ \cite{Efron1979,Efron2014}.

Within this context, given the necessity of calculating $\hth_1, ..., \hth_B$, it is natural to instead consider the estimator
\begin{equation}
	\label{eqn:bagest}
	\tilth_B = \frac{1}{B} \sum_{b=1}^{B} \hth_b
\end{equation}
\noindent as a ``bootstrap smoothed'' alternative to the original $\hth$ \cite{EfronBootstrap}.  This sort of bootstrap aggregation (bagging) was also proposed by \cite{bagging} as a means by which predictive variance may be reduced when each bootstrap sample is used to construct an individual model, frequently a classification or regression tree.

The standard bootstrap approach -- referred to recently as the \emph{brute force} approach by \cite{Efron2014} -- to assessing the variability of $\tilth_B$, though straightforward, is computationally quite burdensome, requiring several bootstrap replicates of not only the original data, but also from within the bootstrap samples themselves.  This double bootstrap, first proposed in \cite{Beran1988}, is especially costly whenever the original statistic $s$ is computationally costly.

A variety of approaches have been suggested to reduce the computational burden of these sorts of problems.  Some of these \citep{White2000,Davidson2000,Davidson2002,Davidson2007} employ what is now referred to as the fast double bootstrap whereby only a single second-level bootstrap sample is collected.  This approach was utilized by \cite{Giacomini2013} in running Monte Carlo experiments and a more careful analysis is given in \cite{Chang2015}.   Sexton and Laake \cite{SextonLaake} propose some alternative nonparametric means by which $\V(\tilth_B)$ may be estimated, suggesting also that the number of second-level bootstrap replicates $B^{'}$ need only be a fraction of the original resample size $B$.  In lieu of full bootstrap samples, subsampling, or $m$-out-of-$n$ bootstrap sampling, was proposed by \cite{Politis1994} and \cite{Bickel1997}.  More recently, \cite{Sengupta2016} proposed a combination of these approaches, first subsampling and then employing a single second-level resample.  Similarly, \cite{Kleiner2014} proposed the bag of little bootstraps which involves splitting the original dataset into a number of subsamples and then taking bootstrap samples on each subset allowing the process to more easily scale by being capable of efficiently running in parallel.

Though the above approaches can substantially reduce the computational complexity in estimating the variance of estimators based on resampling procedures, each nonetheless involves further resampling in order to obtain such an estimate.  Recently, motivated by the problem of taking into account not only the sampling variability but also the variability in model selection, Efron \cite{Efron2014} alleviated this issue by developing an algebraically compact, closed-form estimator for the variance of a bagged estimate.  Instead of additional resampling, Efron's proposal required only additional bookkeeping to recall which samples in the original data appeared how many times in each bootstrap sample.  This development was particularly beneficial in estimating the variance in predictions generated via supervised learning ensembles that are relatively computationally expensive.  A number of recent works, for example, have successfully applied this type of estimator in the context of variance estimation for supervised learning ensembles like random forests \cite{WagerIJ,wager2018estimation,zhou2019asymptotic}.

Though its final form is algebraically simple, the derivation of Efron's variance estimator is fairly involved and may appear somewhat mysterious to many readers.  Its development comes from an application of the original theory for the infinitesimal jackknife involving functional derivatives.  Likely as a result, studies rigorously investigating the statistical properties of this estimator as well as the contexts in which such an estimator would be appropriate are lacking in the current literature.  Efron, for example, notes that the appropriateness of his nonparametric delta method (infinitesimal jackknife) approach follows from the fact that the bagged estimates represent a more smooth function of the data.  Thus, while clearly an extremely significant result in and of itself, these estimates would not appear to apply to more general resampling schemes wherein such smoothness assumptions may not be reasonable.  

In this paper, we strive to take a step forward both in better understanding the intuition behind this important estimator as well as in understanding its statistical properties.  In addition to the infinitesimal jackknife derivation utilized by Efron, we consider two alternative approaches that are a bit more straightforward and easily motivated.  The first of these exploits the important fact that conditional on the observed data, the bagged estimate in \cref{eqn:bagest} depends only on the resampling weights.  We consider a linear approximation to this function of bootstrap weights (i.e.\ standard linear regression) and demonstrate that this approach exactly reproduces the infinitesimal jackknife results given in \cite{Efron2014} whenever all bootstrap samples are employed.  As an additional benefit, this setup motivates a more general procedure for estimating the variance of any resampled estimate, not just one based on the bootstrap. 

In addition to the linear regression and infinitesimal jackknife approaches, we also consider a classical jackknife motivation and once again demonstrate its equivalence in the full bootstrap context.  Importantly, this alternative representation of the estimator allows us to explore its asymptotic properties and in particular, the bias.  While the variance estimators motivated by the jackknife, infinitesimal jackknife, and linear regression approaches are shown to be identical when all bootstrap samples are used, they differ in practical settings when only a randomly selected subsample are employed, suggesting different bias corrections that might be imposed. 

Finally, we derive the form of the infinitesimal jackknife estimate of variance in the U-statistic regime where the resampling is instead done by subsampling without replacement.  We discover that the variance estimators commonly employed in practice for quantifying the predictive uncertainty in supervised learning ensembles are actually something of a ``pseudo" infinitesimal jackknife in that they differ from the ``correct" form when properly derived. 
While this difference is minor when the subsample size is small, in practical modern settings where large subsamples (possibly growing with $n$) are typically employed, such deviations become increasingly significant. We then investigate the properties of the pseudo infinitesimal jackknife and provide a consistent estimate of the variance of (generalized )U-statistics. 

The remainder this paper proceeds as follows. In Section 2, we provide the background and historical motivation for the infinitesimal jackknife (IJ) method. In Section 3, in addition to Efron's recent derivation in \cite{Efron2014}, we provide three alternative approaches for obtaining variance estimates for any bootstrap smoothed statistic including via both the classical jackknife as well as straightforward linear regression.  We discuss the bias of the estimators here and demonstrate their equivalence when all bootstrap samples are employed. Finally, in Section 4, we derive the IJ estimator for the variance of (generalized) U-statistics and discuss its consistency.  To our knowledge, these results provide the first formal guarantees for the validity of confidence intervals constructed on such estimators in settings where all subsamples are not employed, and therefore apply directly to the kinds of modern supervised learning ensembles like random forests commonly employed in practice.


\section{Background of the Infinitesimal Jackknife (IJ)}
\label{sec:background}

Let $\D_n$ denote a sample of observed values from real-valued random variables $X_1,\dots, X_n$ that are i.i.d. from a distribution $\mathrm{P}$. In practice, we are often interested in estimating statistical functionals -- functions of the underlying distribution $\mathrm{P}$,  often estimated via the empirical distribution $\P_n$.  Denote this statistic as $s(X_1,\dots, X_n) = f(\P_n)$ and assume that $s$ is permutation symmetric in these $n$ arguments.  These ``functions of functions" were first introduced by  \cite{volterra1887} and today are a familiar topic of advanced analysis. Any statistic that treats the samples equivalently can also be viewed as a function of $\P_n$, albeit without always necessarily having an explicit form of $f$.  We can further extend the domain of $f$ to any non-negative functions on $X_1,\dots, X_n$ by defining 
\begin{equation}
	f(\P) = f(c\cdot \P), \quad \text{for any } c>0.
\end{equation}

A common task, especially in today's big data era, is to find an appropriate and feasible means of estimating the variance of $f(\P_n)$.  Historically, there have been three primary methods:  the infinitesimal jackknife \cite{Miller1974}, influence curves \cite{hampel1974influence,huber1972}, and the delta method \cite{efron1982jackknife}. Though each method was motivated differently,  \cite{efron1981Nonparametric}  pointed out that the three methods are identical. We thus refer to the common estimator as $\mathrm{IJ}$, which is defined as 
\begin{equation}
	\label{eqn:IJdef}
	\mathrm{IJ} =\frac{1}{n^2} \sum_i\rd_i^2,
\end{equation}
where
\begin{equation}
	\label{di}
	\rd_i = \lim_{\epsilon \to 0} \frac{f((1-\epsilon)\P_n+\epsilon \delta_{X_i}) - f(\P_n)}{\epsilon}
\end{equation}
and $\delta_{x}$ is the Dirac delta function.

\vspace{3mm}

We now briefly review the original derivation of the IJ, following closely to the original constructions given by von Mises in 1947 \cite{mises1947} and Jaeckel in 1972 \cite{Jaeckel1972}.  Let $\mathcal{P}$ be the set of all linear combinations of $\mathrm{P}$ and an arbitrary finite number of the $\delta_x$ measures. Let $\mathcal{P}^+$ be the the set of positive measures in $\mathcal{P}$, not including the zero measures and assume $f$ is defined for the probability measures in $\mathcal{P}^+$.  As above, extend $f$ to all of $\mathcal{P}^+$ by letting $f(c\cdot \rp) =f(\rp)$ for all $c>0$. Note that $\mathcal{P}^+$ is convex and includes $\P_n$. 

We now formally define the derivative of $f$. We say $f$ is differentiable at $\rg$ in $\mathcal{P}^+$ if there exists a function $f'(\rg,x)$, defined at all $x$ in $\R$, with the following property: 

\begin{definition}[ \cite{Jaeckel1972}]
	Let $\rh$ be any member of $\mathcal{P}$ such that $\rg+t\rh$ is in $\mathcal{P}^+$ for all $t$ in some interval $0\leq t\leq t_{\rh}$, $t_{\rh}>0$, so that $f(\rg+t\rh)$ is defined for $t$ in this interval. Then for any such $\rh$, $f'(\rg,x)$ satisfies 
	\begin{equation}
		\begin{aligned}
			\frac{df(\rg+t\rh)}{dt}|_{t=0} &  :=  \lim_{t\to 0}\frac{f(\rg+t\rh)-f(\rg)}{t} \\
			& = \int f'(\rg,x)\,d\rh(x).
		\end{aligned}
	\end{equation}
\end{definition}

\vspace{4mm}

If $\rh=\rg$, we see that   $\int f'(\rg,x)\,d\rg(x) =0$ since $f(c\rg)=f(\rg)$. On the other hand, if $\rh=\delta_{x}- \rg$, we find
\begin{equation}
	\label{ic}
	\lim_{t\to 0}\frac{f((1-t)\rg+t\delta_{x})-f(\rg)}{t} = \int f'(\rg,x)\, d(\delta_x-\rg)(x) = f'(\rg,x).
\end{equation}
Indeed, \cite{hampel1974influence} defined $f'(\rg, x)$ by  \cref{ic} and has called it the ``influence curve", since it reflects the influence of $f$ by adding a small mass on $\rg$ at $x$.  Additionally, the derivative of $f(\rg+t\rh)$ at arbitrary $t_0$ with $0<t_0<t_\rh$ is given by
\begin{equation}
	\label{seg}
	\frac{df(\rg + t\rh)}{dt}|_{t= t_0} = \frac{df(\rg+t_0\rh + u\rh)}{du}|_{u=0} =\int f'(\rg+t_0\rh,x)\,d\rh(x).
\end{equation}

Now assume that $f$ is differentiable (in the sense defined above) at all $\rg$ in some convex neighbor of $\rp$ in $\mathcal{P}^+$ such that $\P_n$ lies in the neighborhood with probability approaching one.  We now describe the motivation for addressing when the $\mathrm{IJ}$ estimator might provide a sensible estimate of the variance of $f(\P_n)$.  Parameterizing the segment from $\rp$ to $\P_n$ by $\rp(t) = \rp+ t(\P_n-\rp)$ for $0\leq t\leq 1$, we have that if $\P_n$ lies in the neighborhood of $\rp$, then
\begin{equation}
	\begin{aligned}
		f(\P_n) - f(\rp) & =  f(\rp(1)) - f(\rp(0)) \\
		& = \frac{d f(\rp(c))}{dt}|_{t=c} \quad \text{(by the Mean Value Theorem)}\\
		& = \int f'(\rp(c))\, d(\P_n - \rp) \quad \text{(by \cref{seg})}
	\end{aligned}
\end{equation}
for some $c$ in $[0,1]$.
Now, for large $n$, $\P_n $ is near $\rp$ and we would expect that $f'(\rp(c))$ is close to $f'(\rp)$, and thus 
\begin{equation}
	\label{app-1}
	\begin{aligned}
		f(\P_n) - f(\rp) & = \int f'(\rp,x)\, d(\P_n-\rp)(x) + o_p\left(\frac{1}{n}\right) \\
		& = \frac{1}{n}\sum_i f'(\rp,X_i)  + o_p\left(\frac{1}{n}\right)
	\end{aligned}
\end{equation}
where the last equality is due to the fact that $\int f'(\rp, x)\, d\rp(x)=0$.
Since the first term on the right-hand side is a sum of i.i.d. random variables, $\sqrt{n}(f(\P_n)-f(\rp))$ is  asymptotic normal with mean $0$ and variance $\mathrm{V} = \int [f'(\rp,x)]^2\,d\rp(x)$. Since $\rp$ is unknown and $f'(\rp,x)$ depends on both $f$ and $\rp$, we generally will not know $f'(\rp, x)$ in advance, and so we could estimate $\rv$ by 
\begin{equation}
	\label{app-2}
	\begin{aligned}
		\rv = \int [f'(\rp, x)]^2\, d\rp(x) &\approx \int f^{'2}(\P_n,x)\,d\P_n(x) \\
		& = \frac{1}{n}\sum_i  [f'(\P_n, X_i)]^2. 
	\end{aligned}
\end{equation}
Then $\V(f(\P_n))$ can be estimated by $\frac{1}{n^2}\sum  [f'(\P_n, X_i)]^2$, which corresponds to the definition of the IJ estimator given in \cref{eqn:IJdef} with $\rd_i = f(\P_n, X_i)$.

In summary, to obtain the final estimate of the variance of $f(\P_n)$, we need to introduce two steps of approximation. First, in \cref{app-1}, we approximate $f(\P_n)$ with a linear statistic at $\P_n$.  Then, in \cref{app-2}, we approximate $\rp$ with $\P_n$ and $f'(\rp)$ with $f'(\P_n)$.  Thus, in evaluating the quality of the IJ estimator, we must determine (i) whether $f$ is close to a linear statistic and (ii) whether $f'(\P_n)$ is close to $f'(\rp)$.  The following sections provide in-depth investigations into these issues for popular types of statistics formed by resampling.

\section{Infinitesimal Jackknife for Bootstrap ($\ij_\rb$)}
\label{sec: bootstrap}

We now focus on a more particular setting -- bootstrapping -- where the infinitesimal jackknife has seen the most success as a method for estimating variance.  Suppose that $s(X_1,..., X_n)$ is statistic, not necessarily a function of $\P_n$, and we take all possible bootstrap samples $(X_1^*,..., X_n^*)$, plug each into $s$ to  obtain a corresponding bootstrap estimate $s^*$, and then take the average. We call the new statistic the bootstrap smoothed (bagged) alternative of $s$ and denote it as $\E_*[s^*]$, where $\E_*[\cdot]$ denotes the expectation taken over the bootstrap sampling procedure conditional on the data. Note that $\E_*[s^*]$ is now a function of $\P_n$. 
%

The dependence of $f$ on $\P_n$ can be explicitly expressed as
\begin{equation}
	f(\P_n) = \int s\, d\P_n\times \cdots \times \P_n = \int s\,d(\P_n)^n
\end{equation}
and so we see that $f$ depends on $(\P_n)^n$ and the dependence roughly exponentially grows with $n$.  According to the Kolmogorov theorem, $\sqrt{n}\sup_x|\P_n(x)-\rp(x)|\to \sup_{t}\rb(\rp(t))$ in distribution, where $\rb(t)$ is the Brownian bridge. Therefore, the distance between $\P_n$ and $\rp$ is on the order of $1/\sqrt{n}$.  Note however that the distance between $(\P_n)^n$ and $(\rp)^n$ is only $\mathcal{O}(1)$ and therefore the distance between $f'(\P_n)$ and $f'(\rp)$ could be large if $f'(\P_n)$ depends on $\P_n$ exponentially.

\subsection{Three Approaches for Variance Estimation}
We turn now to the question of how to derive an estimator for $\V(\E_*[s^*])$.  In the following subsections, we lay out three different approaches, including the Efron's infinitesimal jackknife formulation and ultimately demonstrate that all three are equivalent when all bootstrap samples are utilized. \\


\noindent \textbf{Method 1:  The Infinitesimal Jackknife Approach:}
Since $\E_*[s^*]$ can be viewed as a function of $\P_n$, estimating $\V(\E_*[s^*])= \V(f(\P_n))$ is a standard problem for the infinitesimal jackknife method and we denote the estimator $\mathrm{IJ_B}$. From the definition of $\E_*[s^*]$, we have 
\begin{equation*}
	\begin{aligned}
		f((1-\epsilon)\P_n+\epsilon\delta_{X_i}) & = n^{-n} \sum\frac{s(X_1^*,\dots, X_n^*)n!}{(w_{1}^*!)(w_{2}^*!)\dots (w_{n}^*!)}\left[(1-\epsilon)^{\sum_{k\neq i} w_{k}^*} (1+(n-1)\epsilon)^{w_{i}^*}\right] \\
		&  =  n^{-n} \sum \frac{s(X_1^*,\dots, X_n^*)n!}{(w_{1}^*!)(w_{2}^*!)\dots (w_{n}^*!)}\left[1+n\epsilon(w_{i}^*-1)\right] + o(\epsilon^2) \\
		& = f(\P_n) + \epsilon n\cov_*(s^*,w_i^*)  +  o(\epsilon^2),
	\end{aligned}
\end{equation*}
where $w_i^* = \#\{j:X_j^*=X_i\}$. By \cref{di}, $\rd_i = n \cov_*(s^*, w_i^*))= n\cov_i$ and thus
\begin{equation}
	\label{IJ-bootstrap}
	\mathrm{IJ_B} =\sum_i \cov_i^2 = \sum_i \cov^2_*(s^*, w_i^*).
\end{equation}

The estimator in \cref{IJ-bootstrap} was first derived in 2014 by Efron \cite{Efron2014} as a straightforward application of the classical infinitesimal jackknife method.  This work, however, did not discuss how well the estimator might be expected to perform in various settings.  In fact, even going back to the original classical infinitesimal jackknife results, there has been little to no general theory for addressing such questions.  One possible explanation for the lack of answers in this area is that the process itself is quite abstract; it involves taking functional derivatives, which obfuscates the connection to classical probability theory.  Thus, we now provide an alternative derivation of the same estimator that avoids the need for functional derivatives. In particular, we formulate an explicit expression of $ f'(\rp,x)$ and a different interpretation of $\cov_*(s^*,w_i^*)$. \\

\noindent \textbf{Method 1a:  An Alternative Infinitesimal Jackknife Derivation:}  First, note that $ t= \E_*[s^*]$ is a symmetric function of $X_1,\dots, X_n$ and therefore there exists an H-decomposition \cite{hoeffding1948central} of $t$. 
To set up this H-decomposition, we first need to introduce following notation for kernels $t^{1}, \dots, t^{n}$ of degrees $1,\dots, n$. These kernels are defined recursively as follows
\begin{equation}
	t^{1}(x_1) = t_1(x_1)
\end{equation}
and 
\begin{equation}
	t^c(x_1,\dots, x_c) = t_c(x_1,x_2,...,x_c) - \sum_{j=1}^c \sum_{i_1,\dots, i_j \in \{1,\dots, c\}} t^j(x_{i_1},\dots, x_{i_j})
\end{equation}
where $t_c(x_1,\dots, x_c) = \E[t(x_1,\dots, x_c, X_{c+1},\dots, X_n)] -\E[t]$. Then we can write $t= \E[t] + \sum_j \sum_{i_1,\dots, i_j}t^j (X_{i_1}, \dots X_{i_j})$.
Now consider using the linear term $l_b = \sum_i t^1(X_i)$ as an approximation of $t - \E[t]$.  Then we would expect 
\begin{equation}
	\label{ap-1}
	\begin{aligned}
		t  - \E[t] =   l_b + o_p\left(\frac{1}{n}\right)
	\end{aligned}
\end{equation} 
and naturally, we could use $\V(l_b)$ as an estimate of $\V(\E_*[s^*])$. Note however that $\V(l_b) = n \int (t^1)^2\,d\rp$ and both $\rp$ and $t_1$ are unknown.  First, $\P_n$ is a good candidate for estimating $\rp$, and we have 
\begin{equation}
	\int (t^1)^2\,d\rp \approx \int (t^1)^2\,d\P_n= \frac{1}{n}\sum_{i=1}^n (t^1(X_i))^2.
\end{equation}
Second, as for $t^1(X_1) = \E[t|X_1] - \E[t]$, $\E[\cdot]$ is again unknown, but we could substitute it with  $\E_*[\cdot]$ to obtain
\begin{equation}
	\label{ap-2}
	\begin{aligned}
		t^1(X_1) &  = \E[t(X_1,\dots, X_n)| X_1] - \E[t] \\
		& \approx  \E_*[t(X_1,X_2^*,\dots, X_n^*)] - \E_*[t] \\
		& =   \E_*[s^*(X_1,X_2^*, \dots, X_n^*)] - \E_*[s^*] \\
		& = e_1 - s_0
	\end{aligned}
\end{equation}
where $e_1=\E_*[s^*(X_1,X_2^*,\dots, X_n^*)]$ and $s_0 = \sum e_i/n$. The second equality in \cref{ap-2} holds because $t = \E_*[s^*]$ is already a bagged estimator.
Putting all of the above approximations together, we have
\begin{equation}
	\widehat{\V(l_b)} = \sum_i (e_i-s_0)^2.
\end{equation}

It's worth pausing here to emphasize how the classical ideas behind the infinitesimal jackknife discussed in the previous sections (\cref{app-1} and \cref{app-2}) coincide with the idea of a linear approximation via the H-Decomposition.  In \cref{app-1}, we approximate $t = f(\P) $ by $f(\rp) + \int  f'(\rp,x)\,d\P_n$ whereas in \cref{ap-1}, we approximate $t$ by $\E[t] + \sum_i t^1(X_i)$.  In \cref{app-2}, we approximate $f'(\rp,X_i)$ with $f'(\P_n, X_i)$; in \cref{ap-2} we approximate $t^1(X_i) = \E[t|X_1 = X_i] -\E[t]$ with $\E_*[t|X_1^* = X_i]- \E_*[t] = e_i-s_0$.  It follows that $f'(\rp, x) = nt^1(x)$ and $f'(\P_n, X_i) = n(e_i-s_0)$.  Finally, this alternative derivation also suggests that $\mathrm{IJ_B}$ may only be appropriate for estimating $\V(\E_*[s^*])$ in limited cases when the approximations in \cref{ap-1} and \cref{ap-2} are reasonable. \\

\noindent \textbf{Method 2:  The Jackknife Approach:}  We now propose an estimator for $\V(\E_*[s^*])$ based on the classical jackknife procedure. Recall that delete-1 jackknife samples are selected by taking the original data vector and deleting one observation from the set. Thus, there are $n$ unique jackknife samples, the $i^{th}$ of which is given by $\D_n[i] = (X_1,\dots, X_{i-1},X_{i+1},\dots, X_n)$. The $i^{th}$ jackknife replicate is then defined as the value of the estimator evaluated on the $i^{th}$ jackknife sample. Letting $t$ denote $\E_*[s^*]$, the jackknife estimate of variance is then defined as
\begin{equation}
	\begin{aligned}
		\mathrm{JK} & = \frac{n-1}{n} \sum (t(\D_n[i])-\bar{t})^2\\
		& \approx n\V_*\left(\sum_i t(\D_n[i])\cdot \one_{\{X_i\text{ is deleted}\}}\right).
	\end{aligned}
\end{equation}
Here, the key idea is that we expect $t(\D_n[i])$ be closed to $t(\D_n)$ and thus use the sample variance $t(\D_n[i])$ to estimate the variance of $t(\D_n)$.  Since those $t(\D_n[1]),\dots, t(\D_n[n])$ are strongly correlated, we scale the sample variance by $n$. Note that $t=\int s(x_1,\dots, x_n)\,d\P_n(x_1)\times \cdots \times \P_n(x_n)$.  Now suppose that we consider fixing the $i^{th}$ position rather than deleting the $i^{th}$ sample.  We would thus replace $t(\D_n[i])$ by 
\begin{equation*}
	\begin{aligned}
		t_{(i,j)} & =\int s(x_1, x_2,\dots,x_{i-1}, X_j, x_{i+1},\dots, x_n)\,d\P_n(x_1)\cdots\P_n(x_{i-1})\times\P_n(x_{i+1})\cdots \P_n(x_n) \\
	\end{aligned}
\end{equation*}
for $ 1\leq i,j\leq n$ and propose 
\begin{equation}
	\label{var}
	\begin{aligned}
		\mathrm{JK_B}
		& = n \V_*\left(\sum_j t_{(i,j)}\one_{\{X_i^*=X_j\}}\right)  = \sum_j (e_j-s_0)^2\\
	\end{aligned}
\end{equation}
as an estimate of the variance of $\E_*[s^*]$. The third equality in \cref{var} is simply due to the fact that  $t_{(i,j)} = e_j$.\\

\noindent \textbf{Method 3:  The OLS Linear Regression Approach:} Recall that in a standard bootstrap procedure, $B$ resamples of size $n$ are sampled uniformly at random (with replacement) from the rows of $\D_n$.  Each observation in the original dataset receives equal weight and so we can write that weight vector as $\bw^* \sim \text{Multinomial}(\frac{1}{n}, \dots, \frac{1}{n})$.  Now realize that conditional on the original data $\D_n$, each bootstrap estimate $s^*$ is a function of only those (empirical) weights $(w_1^*,\dots, w_n^*)$ corresponding to the observations actually selected in the bootstrap sample.  Consider the linear space spanned by $\bw^*=(w_1^*,\dots, w_n^*)$ and denote the $l^*$ as the projection of $s(\bw^*)$ onto that linear space.  We can use $\V_*(l^*)$ as an estimate of $\V(\E_*[s^*])$ which corresponds exactly to the setup where the relationship between the bootstrap estimates and observation weights is estimated via ordinary least squares linear regression.   

\begin{remark}
	In the sections that follow, we rely more heavily on Methods 1 and 2 for assessing the properties of the infinitesimal jackknife variance estimator.  Nonetheless, the straightforward OLS approach in Method 3 reveals helpful and important insight into how the infinitesimal jackknife can be viewed.  A much more explicit walkthrough of how the familiar linear regression setup can be used to derive the infinitesimal jackknife estimator is provided in \cref{app:OLSconnection}.
\end{remark}

\subsection{Comparing the Three Approaches}	
We now examine how the three estimators derived above compare to each other.  The first result below gives that the three estimators are identical whenever all bootstrap samples are used.
\begin{theorem}
	\label{3rivers}
	Suppose that we have data $\D_n$ and a statistic $s$. Let $(X_1^*, \dots, X_n^*)$ be a general bootstrap sample of $\D_n$, $s^*=s(X_1^*,\dots, X_n^*)$ and $w_j^*=\#\{i:X_i^*=X_j\}$, then
	\begin{enumerate}
		\item $\E_*[s^*w_j^*] = e_j$;
		\item $l^* =    \sum_j w_j^*\beta_j$, where $\beta_j = (e_j-s_0)$;
		\item 	$\V_*(l^*) = \mathrm{JK_B}  = \mathrm{IJ_B}$.
	\end{enumerate}
	where $e_j=\E_*[s^*|X_1^*=X_j]$ and $s_0=\E_*[s^*]$.
\end{theorem}

The proof of \cref{3rivers} can be found in \cref{app:bootstrap}. 

Now consider the more practical setting in which we draw only $B$ bootstrap samples $(X^*_{b1},\dots, X^*_{bn})$ and calculate $s_b^*= s(X_{b1}^*, \dots, X_{bn}^*)$ for each $b=1,\dots, B$.  Consider the bagged estimate $ \overline{s^*} = \frac{1}{B}\sum_{b=1}^B s_b^*$. By law of total variance, we have 
\begin{equation}
	\label{dec}
	\begin{aligned}
		\V(\overline{s^*}) & = \V(\E[\overline{s^*}|\D_n])  +  \E[\V(\overline{s^*}|\D_n)]  \\
		&= \V(\E_*[s^*]) + \frac{1}{B}\E[\V_*(s^*)].
	\end{aligned}
\end{equation}

\noindent The first term in \cref{dec} is dominant and so to provide a good estimate of $\V(\overline{s^*})$, we must provide a good estimate of $\V(\E_*[s^*])$. Since we do not employ all bootstrap samples, we cannot use $\ij_\rb$ directly but we can use the $B$ bootstrap samples to provide an estimate. First, a natural estimate of $\cov_j$ is simply $\widehat{\cov}_j$, the sample covariance of $(s_1^*, \dots, s_B^*)$ and $(w_{1j}^*, \dots, w_{Bj}^*)$.  Next, for $e_j-s_0$, we can first estimate $s_0$ with $\sum_{b=1}^B s_b^*/B$. Then, since $e_j$ the expected value of $s^*$ given $X_i^*=X_j$ for $i = 1,\dots, n$, a natural estimate is the weighted average of the mean of $s_b^*$ where $X_i^*=X_j$. Here, the weights are the proportion of times when $X_i^* = X_j$ across the $B$ bootstrap samples. A straightforward calculation gives that
\begin{equation}
	\widehat{e_j} = \sum_{b=1}^B\frac{w_{bj}^*}{\sum w_{bj}^*} s_b^*  \quad\text{and}\quad  \widehat{e_j}-\widehat{s_0} =  \sum_{b=1}^B
	\left(\frac{w^*_{bj} }{\sum w^*_{bj}} - \frac{1}{B}\right) s_b^*.
\end{equation}

Finally, the natural estimate of $\V_*(l^*)$ is given by $\widehat{\V}(\hat{l}) $, where $\hat{l}= (\hat{l}_1, \dots, \hat{l}_B)$ is the projection of $(s_1^*,\dots, s_B^*)$ onto the linear space spanned by $(w_{1j}^*, \dots, w_{Bj}^*)$ for $j=1,\dots, n$ and $\widehat{\V}(\hat{l}) = \frac{1}{B}\sum_{b=1}^B ({\hat{l}_b - \bar{\hat{l}}})^2$ is the sample variance of $\hat{l}$.

Putting all of this together, we have the following three finite sample (i.e.\ not all bootstrap samples employed) variance estimators corresponding to the infinitesimal jackknife, jackknife, and OLS linear regression methods: 

\begin{equation}
	\label{mc}
	\begin{aligned}
		\hat{\sigma}^{2}_{\text{IJ}} &:= \widehat{\mathrm{IJ}}_{\mathrm{B}} = \sum_{j=1}^n \widehat{\cov}_j^2 = \sum_{j=1}^n\left[ \frac{1}{B-1} \sum_{b=1}^B (s_b^*-\bar{s^*}) (w^*_{bj}-\bar{w^*_{j}})\right]^2 \\
		\hat{\sigma}^{2}_{\text{JK}} &:= \widehat{\mathrm{JK}}_\mathrm{B} =\sum_{j=1}^n (\widehat{e_j}-\widehat{s_0})^2 = \sum_{j=1}^n \left[\sum_{b=1}^B
		\left(\frac{w^*_{bj} }{\sum w^*_{bj}} - \frac{1}{B}\right) s_b^*\right]^2\\
		\hat{\sigma}^{2}_{\text{OLS}} &:= \widehat{\V}(\hat{l}) = \frac{1}{B}\sum_{b=1}^B ({\hat{l}_b - \bar{\hat{l}}})^2. \\
	\end{aligned}
\end{equation}

While these alternative derivations and estimators are helpful for providing insight into the infinitesimal jackknife in general, $\hat{\sigma}^{2}_{\text{IJ}} = \widehat{\mathrm{IJ}}_{\mathrm{B}}$ seems to be the estimator that has garnered the most practical interest.  In the following subsections, we focus on assessing its properties.

\subsection{The Bias of $\widehat{\ij}_\rb$}
\label{sec:biasIJ}

We now begin to analyze the properties of $\hij_{\mathrm{B}}$ as an estimator.  As discussed above in reference to \cref{dec}, in order to understand how well $\hij_{\mathrm{B}}$ estimates $\V(\overline{s^*})$, we need only understand how well it estimates $\V(\E_*[s^*])$.  Here we consider both the Monte Carlo bias and sampling bias of $\widehat{\ij}_{\mathrm{B}}$ where the sampling bias is considered with respect to variation in the data, whereas the Monte Carlo bias is considered with respect to bootstrap process conditional on the data. We combine these two sources of bias with
\begin{equation}
	\widehat{\mathrm{IJ}}_\rb - \V(\E_*[s^*])   \propto  \underbrace{\E_*[\widehat{\mathrm{IJ}}_\rb] -\mathrm{IJ}_\rb}_{\textbf{Monte Carlo Bias}} + \underbrace{ \E[\mathrm{IJ}_\rb] - \V(\E_*[s^*])}_{\textbf{Sampling Bias}}.
\end{equation}

\noindent and consider their impact separately in the following subsections. \\

\subsubsection{Monte Carlo Bias}
We first consider the Monte Carlo bias of  $\widehat{\mathrm{IJ}}_\rb$.  Note that $\E_*[\hij_\rb] -\ij_\rb = \sum_j \E_*[\widehat{\cov}_j^2] - \cov_j^2$ and $\E_*[\widehat{\cov}_j^2 ] - \cov_j^2 
= \E_*[\widehat{\cov}_j^2 ] - \E_*^2[\widehat{\cov}_j]   = \V_*(\widehat{\cov}_j)$. Next, we have 
\begin{equation*}
	\begin{aligned}
		\V_*(\widehat{\cov}_j) 
		&  =  -\frac{B-2}{B(B-1)}\cov^2_*(s^*,w_j^*)  + \frac{\V_*(s^*)\V_*(w_j^*)}{B(B-1)} +\\ & ~~~~\frac{\E_*[(s^*-\E_*[s^*])^2(w_j^*-\E_*[w_j^*])]^2}{B} \\
		& := \mathrm{I} + \mathrm{II}
	\end{aligned}
\end{equation*}
where 
\begin{equation}
	\begin{aligned}
		\mathrm{I} &= \frac{1}{B}\V_*\left((s^*-\E_*[s^*])(w_j^*-\E_*[w_j^*])\right) \\
		\mathrm{II} & = \frac{1}{B(B-1)}\left[\V_*(s^*)\V_*(w_j^*) + \cov_*^2(s^*,w_j^*)\right].
	\end{aligned}
\end{equation}
The first term $\mathrm{I}$ is the dominant term and is $\mathcal{O}(1/B)$. Essentially, we see here that using $\widehat{\cov}_j^2$ to estimate $\cov_j^2$ is analogous to using $\bar{X}^2$ to estimate $\E^2[X]$, which is biased since $\E[\bar{X}^2] - \E^2[X] = \V(X)/B$.  Indeed, $\V(X)/B$ may not be negligible, especially when $B$ is small and  the coefficient of variation of $X$ is large. The variance estimator defined below offers a bias correction for $\hij_\rb$.

\begin{definition}
	A Monte Carlo bias corrected version of $\hij_\rb$ is given by
	\begin{equation}
		\label{mc-ij}
		\hij_\rb^{mc} = \hij_\rb -  \frac{1}{B} \sum_j \widehat{\V}\left((s^*-\bar{s^*})(w_j^*-\bar{w_j^*})\right),
	\end{equation}
	where $\widehat{\V}$ denotes sample variance. 
\end{definition}

Note that the bias correction term above is a sum over $n$ terms. Then if $B$ is small, the bias correction term will be significant. 
\begin{remark}
	In recent work, \cite{WagerIJ} proposed the Monte Carlo bias corrected estimator
	\begin{equation}
		\label{mc-efron}
		\hij_\rb^{whe}= \hij_\rb-  \frac{n}{B}\widehat{\V}(s^*).
	\end{equation}
	We see that if $\V_*((s^*-\E_*[s^*])(w_j^*-\E_*[w_j^*])$ is close to $\V_*(s^*-\E_*[s^*])\V_*(w_j^*-\E^*[w_j^*]) = (1-\frac{1}{n})\V_*(s^*-\E_*[s^*])$, then \cref{mc-ij} is close to \cref{mc-efron}.  Simulations comparing \cref{mc-ij} and  \cref{mc-efron} are provided in the next section. \\
\end{remark}


\subsubsection{Sampling Bias}
Before looking more generally at the sampling bias of the infinitesimal jackknife, let's first examine how $\ij_\rb$ behaves on some simple examples.

\vspace{3mm}

\noindent \textbf{Example 1: Sample Mean}
Consider $s=s(x_1,\dots, x_n) = \frac{1}{n}\sum_{i=1}^n x_i$. We have 
\begin{equation}
	s^* = \frac{1}{n}\sum_{i=1}^n X_i^*, \quad l^* = \sum_{i=1}^n \sum_{j=1}^n 1_{X_i^*=X_j} (e_j-s_0).
\end{equation}
Then, $\E_*[s^*] = \frac{1}{n}\sum_{i=1}^n X_i$ and  $\V_*(l^*) =\frac{1}{n^2}\sum_{i=1}^n (X_i-\bar{X})^2$. Therefore, 
\begin{equation}
	\V(\E_*[s^*]) = \sigma^2/n, \quad \E[\V_*(l^*)] = (n-1)\sigma^2/ n^2
\end{equation}
and thus, we have  $ \frac{\E[\V_*(l^*)]}{\V(\E_*[s^*])} = \frac{n-1}{n}\to 1$ as $n\to \infty$. In  Figures 1 and 2, $X_1,\dots, X_n$ follow $\mathcal{N}(0, \sigma^2)$, $n=100$ and $\sigma^2 =1$. Since we know that $\V(\E_*[s^*])=\sigma^2/n$ , an oracle estimate would be $\widehat{\sigma^2}/n$, where $\widehat{\sigma^2}$ is the sample variance. The gray dashed line denotes the true value of $\V(\E_*[s^*])$.  We find that $\hij_\rb^{mc}$ and $\hij_\rb^{whe}$ are quite close as expected and both perform well. The original $\hij_\rb$ seems to overestimate substantially when $B=100$.

\begin{figure}[h]
	\centering
	\begin{minipage}{0.45\linewidth}
		\label{bpmean}
		\includegraphics[width=\textwidth,height=\textheight,keepaspectratio]{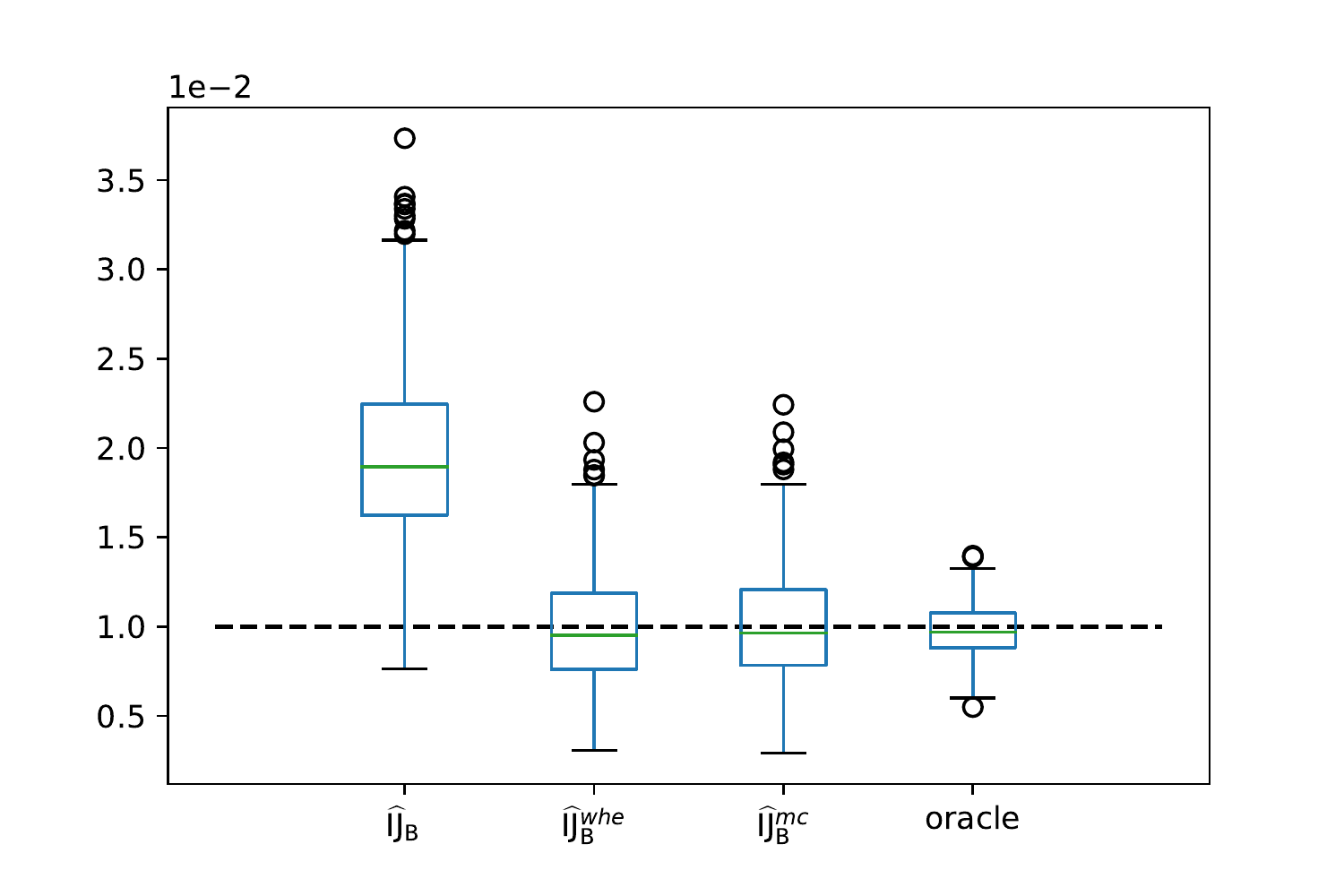}
		\caption{Performance of the infinitesimal jackknife and its bias-corrected alternatives on estimating the variance of the  bagged sample mean (B=100).}
	\end{minipage}
	\quad 
	\begin{minipage}{0.45\linewidth}
		\includegraphics[width=\textwidth,height=\textheight,keepaspectratio]{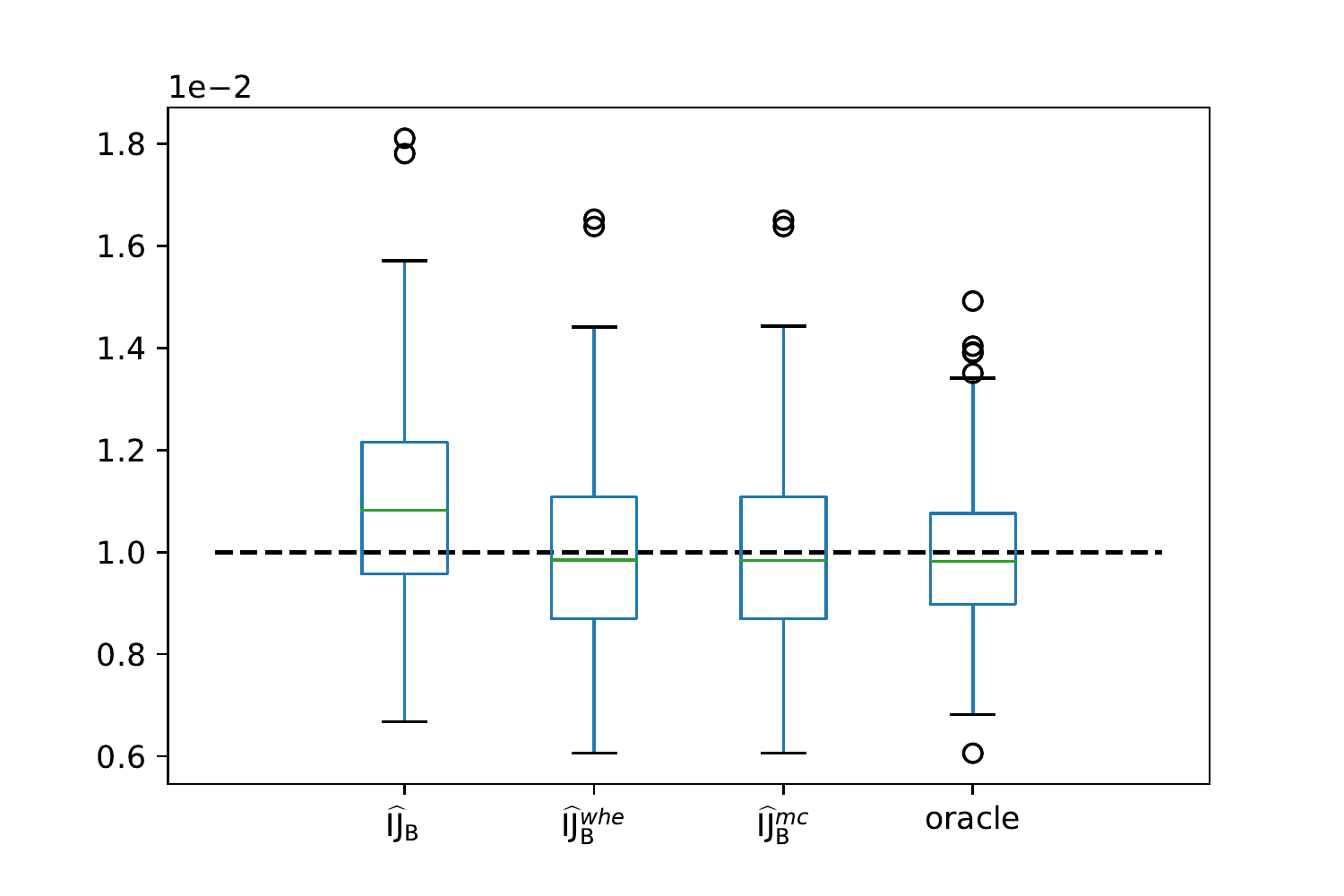}
		\caption{Performance of the infinitesimal jackknife and its bias-corrected alternatives on estimating the variance of the  bagged sample mean (B=1000).}
	\end{minipage}
\end{figure}

\vspace{3mm}
\noindent \textbf{Example 2: Sample Variance}
Consider  $s= {n \choose 2}^{-1}\sum \limits_{i<j} (x_i-x_j)^2$. We have 
\begin{equation*}
	\E_*[s^*] = \frac{1}{n}\sum_{i=1}^n(X_i-\bar{X})^2\quad \text{and}\quad \V_*(l^*) = \frac{1}{n^2} \sum_{i}\left[(X_i-\bar{X})^2- \frac{1}{n}\sum_{i}(X_i-\bar{X})^2\right]^2.
\end{equation*}
Then we have 
\begin{equation}
	\begin{aligned}
		\V(\E_*[s^*]) & =\left(\frac{n-1}{n}\right)^2\left[\frac{\mu_4}{n} - \frac{\mu_2^2}{n}\frac{n-3}{n-1}\right] \\
		& = a_n\mu_4 - b_n\mu_2^2,
	\end{aligned}
\end{equation}
where $\mu_i$ is the $i$th central moment of $X$ for $i=2,4$. Let $\bX= (X_1,\dots, X_n)^T$, then 
$\E[\V_*(l^*)] $ can be written as  $ \frac{1}{n} \E[\bX^T\A\bX]^2$,
where  $\A = \Sigma_1 - \frac{1}{n}\sum_i \Sigma_i$, $\Sigma_i = (e_i-\frac{1}{n}\bone_n)(e_i^T -\frac{1}{n}\bone_n^T)$ and $e_i=(0,\dots, 0,1,0,\dots, 0)$. After some calculation, we obtain
\begin{equation*}
	\begin{aligned}
		& ~~~~ \E[\V_*(l^*)]  \\
		&= \frac{(n-1)}{n^2}\left[\E[(X_1-\bar{X})^4] -\E[(X_1-\bar{X})^2(X_2-\bar{X})^2]\right] \\
		& = \left(\frac{n-1}{n}\right)^2 \left[ \left(\frac{n^3-(n-1)^2}{n^2(n-1)^2} + \frac{n}{(n-1)^5}\right)\mu_4 - \left(\frac{n^2-2n+3}{(n-1)n^2} - \frac{3n^2(2n-3)}{(n-1)^5}\right)\mu_2^2\right]\\
		& = a_n'\mu_4 - b_n' \mu_2^2.
	\end{aligned}
\end{equation*}
Thus, we have 
\begin{align}
	\nonumber
	\frac{a_n'}{a_n} &=   1 + \frac{n^2+n-1}{n(n-1)^2} + \frac{n^2}{(n-1)^5}  = 1 + \frac{1}{n} + o(\frac{1}{n}) \\
	\nonumber
	\frac{b_n'}{b_n} & = 1 + \frac{n+3}{n(n-3)}  -  \frac{3n^3(2n-3)}{(n-1)^4(n-3)} = 1- \frac{5}{n} + o(\frac{1}{n}).
\end{align}
Since $a_n'/a_n\to 1 $ and $b_n'/b_n\to 1$ as $n\to \infty$, we have $\frac{\E[\V_*(l^*)]}{\V(\E_*[s^*])} \to 1$.  $\ij_\rb$ is therefore asymptotically unbiased for estimating  the variance of the sample variance. Since the sample variance is close to a linear statistic, the result is not surprising. In Figures 3 and 4,  $X_1,\dots, X_n$ follow $\mathcal{N}(0, \sigma^2)$, $n=100$,  and $\sigma^2 =1$. Since we know $\V(\E_*[s^*]) = 2\sigma^4/n$, an oracle estimate would be $2(\widehat{\sigma^2})^2/n$, where $\widehat{\sigma^2}$ is  the sample variance. The gray dashed line denotes the true value of $\V(\E_*[s^*])$. As in the first example, $\hij_\rb^{mc}$ and $\hij_\rb^{whe}$ are both quite close and perform well. The original $\hij_\rb$ again seems to suffer from overestimation when $B=100$.

\begin{figure}[h]
	\centering
	\begin{minipage}{0.45\linewidth}
		\includegraphics[width=\textwidth,height=\textheight,keepaspectratio]{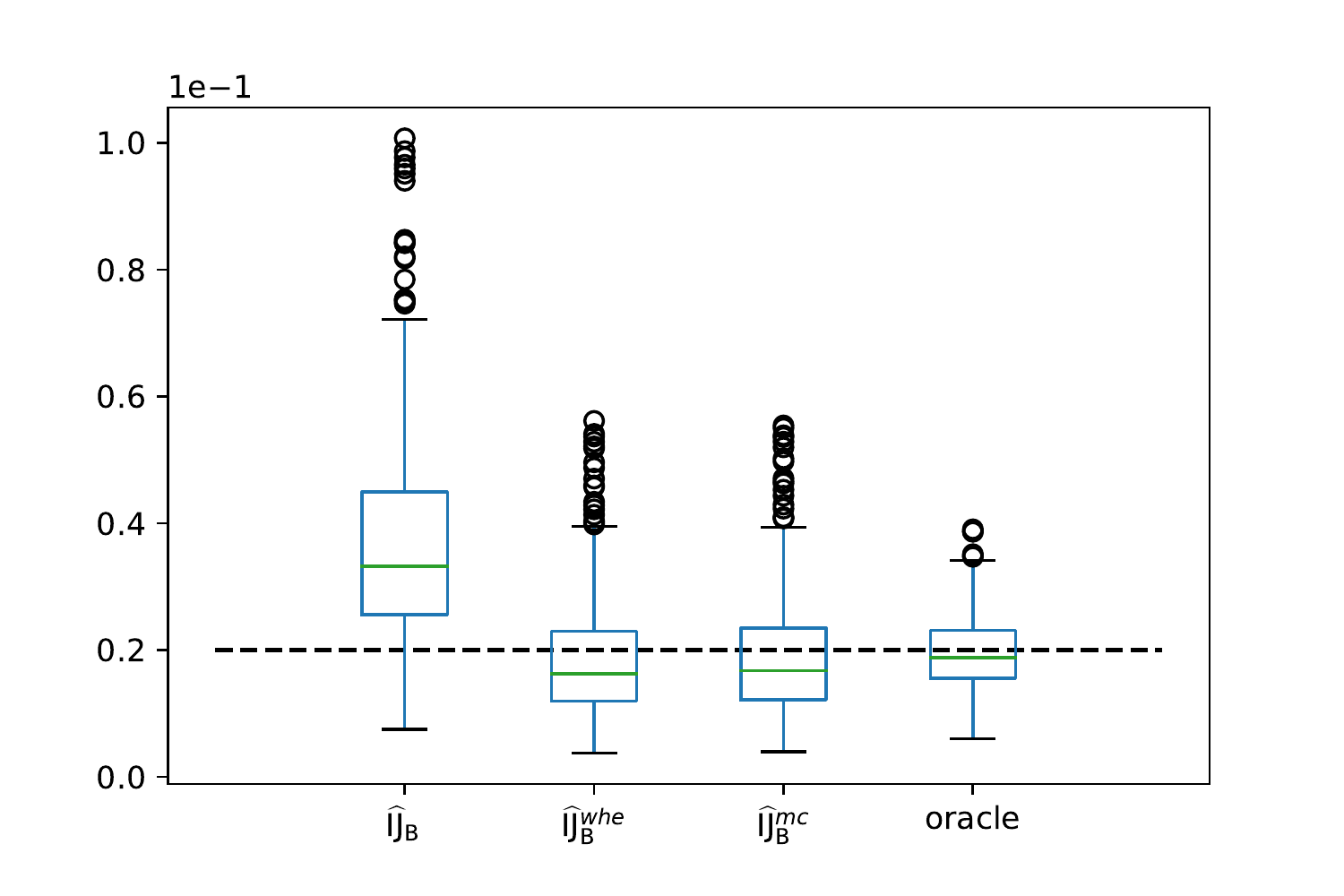}
		\caption{Performance of the infinitesimal jackknife and its bias-corrected alternatives on estimating the variance of the  bagged sample variance (B=100).}
	\end{minipage}
	\quad 
	\begin{minipage}{0.45\linewidth}
		\includegraphics[width=\textwidth,height=\textheight,keepaspectratio]{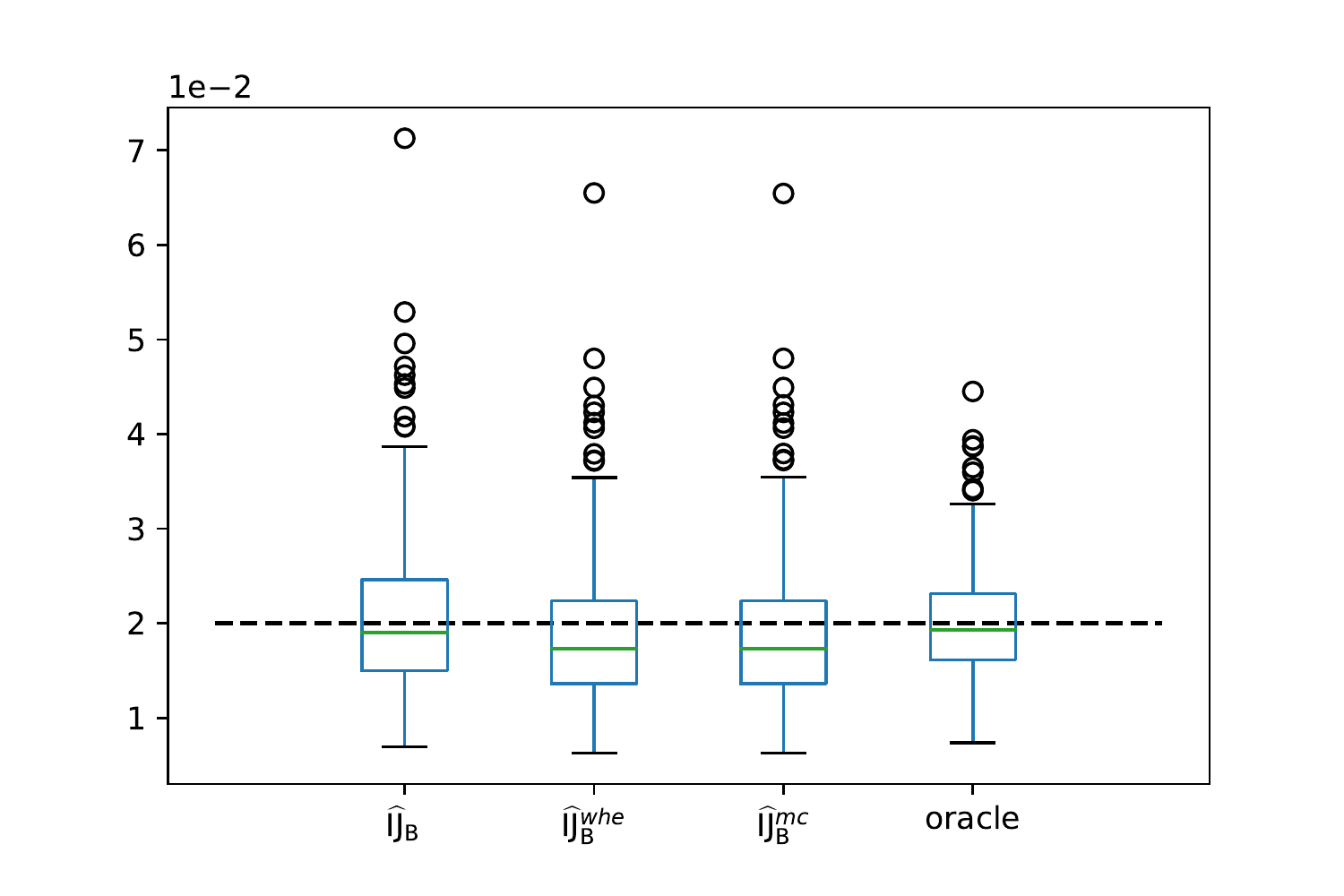}
		\caption{Performance of the infinitesimal jackknife and its bias-corrected alternatives on estimating the variance of the  bagged sample variance (B=1000).}
	\end{minipage}
\end{figure}

\vspace{3mm}
\noindent \textbf{Example 3: Sample Maximum}
Consider $s=\max \{X_1,\dots, X_n\}$, where $X_1,\dots, X_n$ are uniformly distributed in $[0,1]$. Then
\begin{equation}
	\V(\E_*[s^*]) = \frac{(n-\sum_{j=1}^n (\frac{j-1}{n})^n) (1+ \sum_{j=1}^n (\frac{j-1}{n})^n)}{(n+1)^2(n+2) }
\end{equation}
and 
\begin{equation}
	\E[\V_*(l^*)] = \frac{\sum_i [ \sum_{j=1}^{n} (\frac{j-1}{n})^n - \sum_{j=i+1}^{n}(\frac{j-1}{n})^{n-1}]^2}{(n+1)(n+2)}. 
\end{equation}
The details of the calculation can be found in \cref{app:bootstrap}. We have
\begin{equation}
	\begin{aligned}
		\frac{\E[\V_*(l^*)]}{\V(\E_*[s^*])} 
		& =  \frac{(n+1)\sum_i [ \sum_{j=1}^{n} (\frac{j-1}{n})^n - \sum_{j=i+1}^{n}(\frac{j-1}{n})^{n-1}]^2}{(n-\sum_j  (\frac{j-1}{n})^n) (1+ \sum_j (\frac{j-1}{n})^n)}\\
		& \to c  \in [0.24,0.25]\quad  (n \to \infty)
	\end{aligned}
\end{equation}
and thus we see that $\ij_\rb$ is underestimating of $\V(\E_*[s^*])$ by a considerable margin. In this case, $\E_*[s^*]$ is not close to a linear statistic, so $\ij_\rb$ should not be expected to perform well. In Figures 5 and 6, $X_1,\dots, X_n$ follow Unif(0, 1) and $n=100$ with the dashed line corresponding to the true value of $\V(\E_*[s^*])$. In this case, there is no obvious oracle estimator for $\E_*[s^*]$.  Unlike the previous two examples, although $\hij_\rb^{mc}$ and $\hij_\rb^{whe}$ remain quite similar, all three estimators suffer from considerable underestimation even when $B=1000$.

\begin{figure}[h]
	\centering
	\begin{minipage}{0.45\linewidth}
		\includegraphics[width=\textwidth,height=\textheight,keepaspectratio]{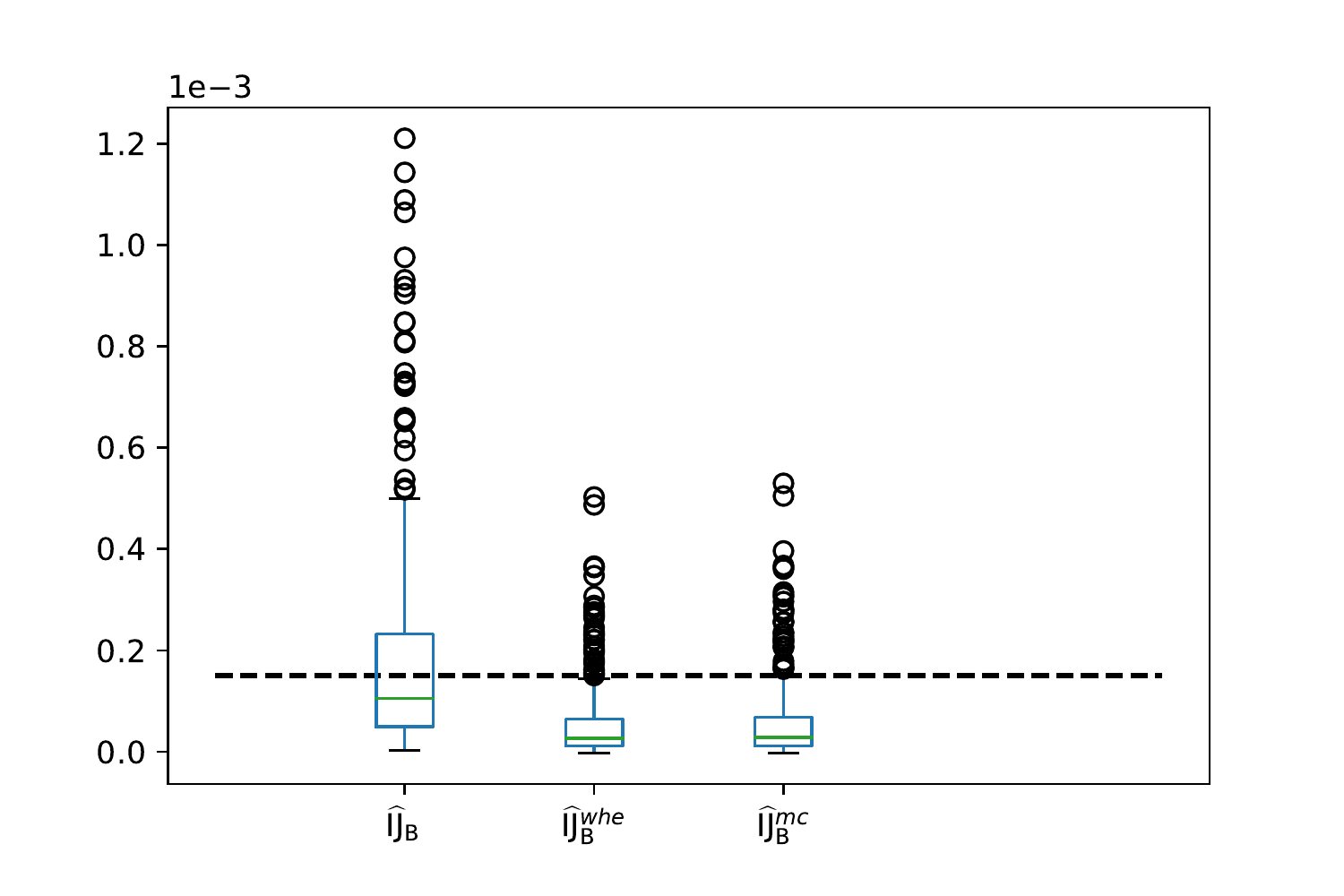}
		\caption{Performance of the infinitesimal jackknife and its bias-corrected alternatives on estimating the variance of the  bagged sample maximum (B=100).}
	\end{minipage}
	\quad 
	\begin{minipage}{0.45\linewidth}
		\includegraphics[width=\textwidth,height=\textheight,keepaspectratio]{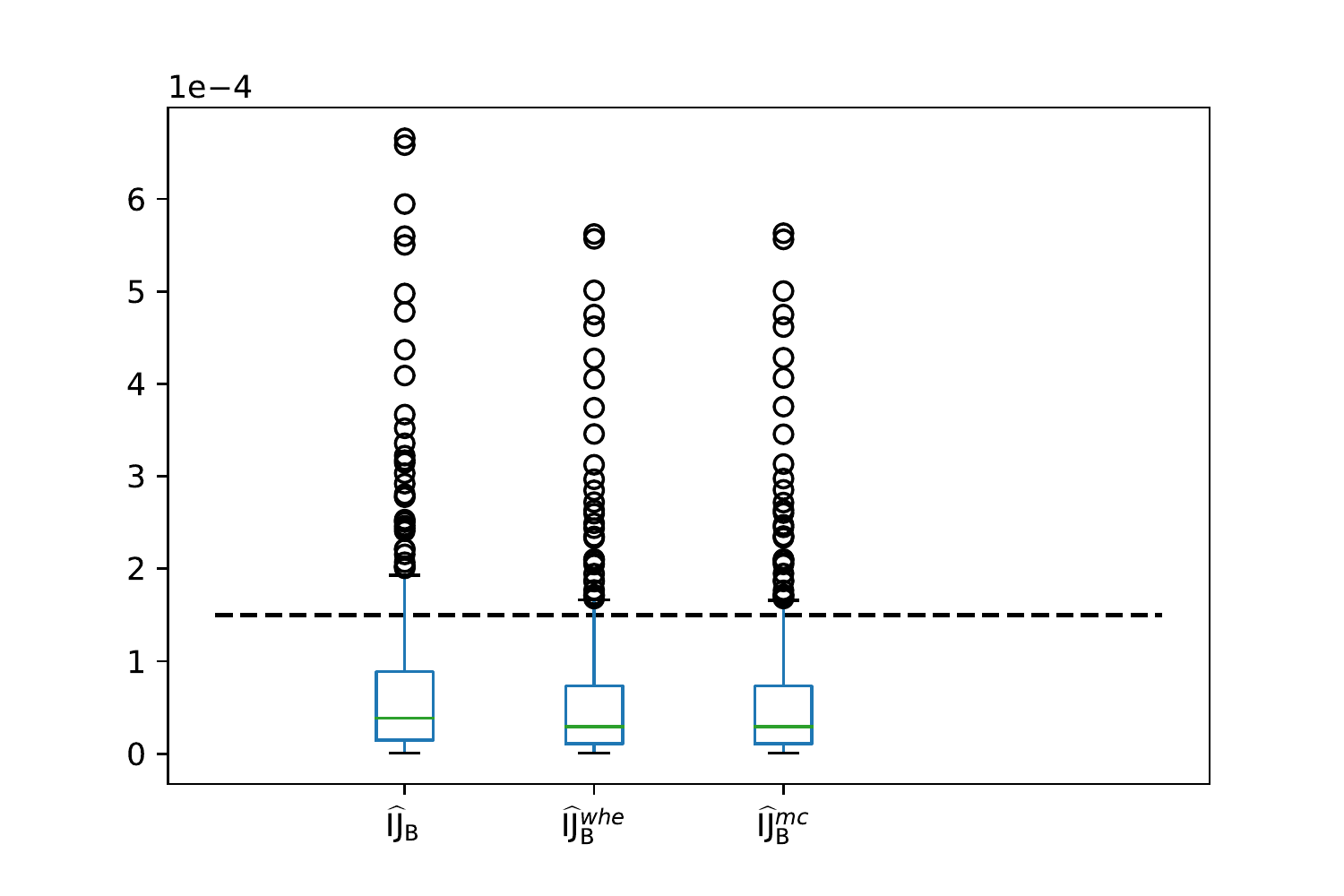}
		\caption{Performance of the infinitesimal jackknife and its bias-corrected alternatives on estimating the variance of the  bagged sample maximum (B=1000).}
	\end{minipage}
\end{figure}

The following theorem provides an equivalent condition for $\ij_\rb$ to be asymptotically unbiased and suggests that $\E_*[s^*]$ needs to be asymptotically linear.

\begin{theorem}
	\label{BOOT-2}
	Let $\E_*[s^*]$ be the bootstrap smoothed alternative of $s$, then 
	\begin{equation}
		\lim_{n\to \infty} \frac{\E[\V_*(l^*)]}{\V(\E_*[s^*])} = 1 \iff  \lim_{n\to \infty} n(1-\rho) = 1,
	\end{equation}
	where $\rho$ is the correlation coefficient between $e_1$ and $e_2$ and $e_i = \E_*[s^*|X_1^*=X_i]$.
\end{theorem}

To help provide some intuition, consider the case where $s=\bar{X}$, which is linear. Here, $\rho= \frac{n^2-1}{n^2+n-1} = 1 -1/n + o(1/n)$ and $\E[\V_*(l^*)] / \V(\E_*[s^*]) = \frac{n-1}{n}\to 1$. We suspect that to make $\rho = 1-1/n + o(1/n)$, $\E_*[s^*]$ is required to be equal to $l_b + o_p(1/n)$. From the original derivative of IJ and  the above examples, we can see that it is fairly hard for $\ij_\rb$ to estimate $\V_*(\E_*[s^*])$ well. In general, it is difficult to understand the bias and other statistical properties of $\ij_\rb$ comprehensively, largely due to the replicates of bootstrap samples.


\section{U-statistics and Infinitesimal Jackknife}
\subsection{Infinitesimal Jackknife for U-statistics ($\ij_\ru$)}

While the infinitesimal jackknife is naturally appealing for estimating the variance of bagged statistics, the current resurgence in interest is due in large part to its potential for quantifying the uncertainty of complex supervised machine learning ensembles often formed by subsampling.  Immediately following Efron's work in \cite{Efron2014}, \cite{WagerIJ} showed that the same infinitesimal jackknife approach could be used to generate confidence intervals for random forests.  In recent years, a number of works have expanded on this idea \cite{Mentch2016,Wager2018,peng2019asymptotic}.  Among the key breakthroughs was the realization that predictions generated by averaging across ensembles of regression estimates, each formed with subsamples of the original data, could be seen as akin to classical U-statistics.

We thus now explore how the infinitesimal jackknife can be extended to the case of subsampling without replacement. In this case,  $\E_*[s^*]$ is a U-statistic, which is often more convenient for theoretical analysis and also more likely to be close to linear.  Here $s$  is a (permutation-symmetric) function of $k$ i.i.d. random variables and the U-statistic can be written as 
\begin{equation}
	\begin{aligned}
		\mathrm{U} & = {n\choose k}^{-1}\sum_{(n,k)} s(X_{i_1}, \dots, X_{i_k}) \\
	\end{aligned}
\end{equation}
where the sum is taken over all ${n\choose k}$ subsamples of size $k$.
How does U depend on $\P_n$, such that $\ru = f(\P_n)$ for some $f$?  The dependence is abstract so that the subsampling proceeds according to the probabilities determined by $\P_n$.  Following directly from the original definition of the IJ, we arrive at the following theorem. 
\begin{theorem}
	\label{U-ij-1}
	The IJ estimator of the variance of a U-statistic is given by
	\begin{equation}
		\label{def:iju}
		\mathrm{IJ}_\ru= \frac{k^2}{n^2}\sum_{j=1}^n[\alpha e_j-\beta s_0]^2,
	\end{equation} 
	where $e_j = \E_*[s^*|X_1^*=X_j]$, $s_0=\E_*[s^*]$ and 
	$$\alpha = 1 + \frac{1}{n}\left\{\frac{k-1}{2} - \frac{1}{k}\sum_{j=0}^{k-1}\frac{j^2}{(n-j)}\right\}, \quad 
	\beta = 1 + \frac{1}{k}\sum_{j=0}^{k-1}\frac{j}{n-j}.$$
\end{theorem}
To understand the bias of $\ij_\ru$, the H-decomposition will prove quite useful. To set up this H-decomposition, we first need to introduce following notation for kernels $s^{1}, \dots, s^{k}$ of degrees $1,\dots, k$. These kernels
are defined recursively as follows
\begin{equation}
	s^1(x_1) = s_1(x_1)
\end{equation}
and 
\begin{equation}
	s^c(x_1,\dots, x_c) = s_c(x_1,x_2,...,x_c) - \sum_{j=1}^c \sum_{i_1,\dots, i_j \in \{1,\dots, c\}} s^j(x_{i_1},\dots, x_{i_j})
\end{equation}
where $s_c(x_1,\dots, x_c) = \E[s(x_1,\dots, x_c, X_{c+1},\dots, X_k)] -\E[s]$. Let $V_j = \V(s^{j})$ for $j=1,\dots, k$.  Then $\E[\ij_\ru]$ can both be written as a linear combination of those $V_j$. In particular, we have the following theorem.
\begin{theorem}
	\label{U-ij-2}
	Let $\theta=\E[s]$ and $\ij_\ru$ be as defined in \cref{def:iju}.  Then 
	\begin{equation}
		\E[\ij_\ru] = \sum_{j=1}^k r_j {k\choose j}^2{n\choose j}^{-1} V_j + \frac{k^2}{n}(\alpha-\beta)^2\theta^2, 
	\end{equation}
	where 
	\begin{equation}
		r_j=   \frac{(n-k)^2}{n^2}\left[\frac{j}{1-j/n} \alpha^2\right]+\frac{k^2}{n} (\alpha-\beta)^2 , \quad \text{for } j=1,\dots,k.
	\end{equation}
	
\end{theorem}

\begin{remark}
	Note that $\V(\ru) = \sum_{j=1}^k {k\choose j}^2{n\choose j}^{-1}V_j$. If $k$ is fixed, then $\alpha, \beta \to 1$ and thus  $r_j\to j$ for $j=1,\dots, k$. Since $\V(\mathrm{U})$ is dominated by the $V_1$ term, $\ij_\ru$ is asymptotically unbiased.  
\end{remark}

\subsection{The Pseudo Infinitesimal Jackknife for U-statistics ($\sij_\ru$)}

In recent work, \cite{Wager2018} investigated the consistency and asymptotic normality of random forests where the individual trees were constructed with subsamples of the original data.  As part of this work, the authors proposed another estimate of variance wherein the format of the IJ for bootstrap was simply copied over to this subsampling regime to arrive at
\begin{equation}
	\label{pseudo-IJ-U}
	\sum_j \cov^2_*(s^*, w_j^*),
\end{equation}
where $*$ refers to the subsampling procedure.  We refer to this estimator as $\sij_\ru$, since it is not derived from the definition of the infinitesimal jackknife.   It is, however, possible to provide a more rigorous motivation for $\sij_\ru$.
Recall from Section 2 (see \cref{app-1}), we assume that $f(\P_n)-f(\rp)$ can be written as $\frac{1}{n}\sum_i f'(\rp,X_i) + o_p(1/n)$, where the dominant term is a sum of i.i.d. random variables and we estimate the variance of $f'(\rp,X_i)$ by $\frac{1}{n}\sum_i f^{'2}(\P_n,X_i)$. Now suppose that we write $f(\P_n) - f(\rp)$ as  $\sum_i g(X_i) + o_p(1/n)$, where $g(X_i)$ is not necessarily $\frac{1}{n}f'(\rp, X_i)$. From classical U-statistic theory, we know that there is a natural candidate for $g(X_i)$:  the H\'ajek projection - $\E[f(\P_n)-f(\rp)|X_i] = \frac{k}{n}\E[s-\E[s]|X_i]$. Since $\V(\sum g(X_i)) = \frac{k^2}{n}V_1$, where $V_1 = \V(\E[s|X_1])$, we need only to propose a reasonable estimate for $V_1$ and we can then use 
\begin{equation}
	\label{l-u}
	\frac{k^2}{n}\hat{V}_1
\end{equation}
as an estimate of the variance of the U-statistic. Since $V_1 = \E[\E[s|X_1]-\E[s]]^2$, a natural candidate of $\hat{V}_1$ would be
\begin{equation}
	\frac{1}{n}\sum_j (\E_*[s^*|X_1^*=X_j] - \E_*[s^*])^2 = \frac{1}{n}\sum_j (e_j-s_0)^2.
\end{equation}
As it turns out that \cref{l-u} is the same as \cref{pseudo-IJ-U}.
\begin{proposition}
	\label{U-sij-pre}
	Let $\D_n^*=(X_1^*,\dots, X_k^*)$ denote a subsample of size $k$ from the original data $\D_n$ and define $w_j^* =  \one_{X_j\in\D_n^* }$.  Then 
	\[\cov_*(s^*, w_j^*) =  \frac{k}{n}(e_j-s_0),\]
	where $*$ refers to the subsampling procedure, $e_j = \E_*[s^*|X_1^*=X_j]$ and $s_0 = \E_*[s^*]$.
	Thus,
	\begin{equation}
		\sij _\ru =  \sum_j \cov^2_*(s^*, w_j^*)= \left(\frac{{k\choose 1}}{{n\choose 1}}\right)^2\sum_j  (e_j-s_0)^2.
	\end{equation}	
\end{proposition}

\noindent In an analogous fashion to $\ij_\ru$, we can establish the following theorem.
\begin{theorem}
	\label{U-sij}
	The pseudo-IJ estimator of the variance of a U-statistic is defined as 
	\begin{equation}
		\sij_\ru= \frac{k^2}{n^2}\sum_{j=1}^n[ e_j-s_0]^2
	\end{equation} 
	where $e_j = \E_*[s^*|X_1^*=X_j]$ and $s_0=\E_*[s^*]$. Then 
	\begin{equation}
		\E[\sij_\ru] = \sum_{j=1}^k r_j{k\choose j}^2{n\choose j}^{-1} V_j,
	\end{equation}
	where 
	\begin{equation}
		r_j = \left(\frac{n-k}{n}\right)^2 \frac{j}{1-j/n}, \quad \textit{for } i = 1,\dots, k.
	\end{equation}
\end{theorem}

Note that although our goal here is to use $\frac{k^2}{n}\hat{V}_1$ to estimate $\frac{k^2}{n}V_1$, $\E[\frac{k^2}{n}\hat{V}_1] = \E[\sij_\ru]$ involves higher order terms of $V_2,\dots, V_k$.  Though not ideal, it is unavoidable since we do not have new data generated from the underlying distribution. If we simply multiply $\sij_\ru$ by $(\frac{n}{n-k})^2\frac{n-1}{n}$ as proposed in \cite{Wager2018}, then only the first term is unbiased, but it doubles the quadratic term, triples the cubic term etc. (a similar phenomenon was discovered by \cite{efron1981Nonparametric} for the jackknife variance estimator). This explains why this estimation procedure is inflated in practice.  In many applications, $k$ is not small and so the higher order terms of $\V(\ru)$ are not negligible and the effect of $r_j$ cannot be ignored. Indeed, as alluded to at the beginning of this section, in many modern machine learning applications like random forests, $k$ corresponds to the number subsamples utilized in the construction of each base learner (trees, in the case of random forests),and so $k$ is generally best chosen to be as large as possible.

\subsection{The Consistency of $\sij_\ru$ and Its Derivatives}

Comparing Theorem's \ref{U-ij-1} and \ref{U-sij}, it is straightforward to see that if $k/n\to 0$, then $\alpha \to 1$ and $\alpha-\beta\to 0$, and thus $\ij_\ru \to \sij_\ru$.  More generally, $\sij_\ru$ has a slighter simpler expression and in this subsection, we investigate its statistical properties, focusing in particular on its consistency.  We begin by introducing the idea of generalized U-statistics, recently defined in \cite{peng2019asymptotic}. 
\begin{definition}[Generalized U-statistic \cite{peng2019asymptotic}]
	\label{gu}
	Suppose $X_1, \dots, X_n$ are i.i.d.\ samples from $\rp$ and let $s$ denote a (possibly randomized) real-valued function that is permutation symmetric in its $k \leq n$ arguments. A generalized U-statistic with kernel $s$ of order (rank) $k$ refers to any estimator of the form 
	\begin{equation}
		\label{eqn:gu}
		\ru_{n,k,N,\omega} = \frac{1}{\hat{N}} \sum_{(n,k)} \rho s(X_{i_1}, \dots, X_{i_k}; \omega)
	\end{equation}
	where $\omega$ denotes i.i.d.\ randomness, independent of the original data. The $\rho$ denotes i.i.d.\ Bernoulli random variables determining which subsamples are selected where $\P(\rho=1) = N/{n\choose k}$ and $\hat{N}$ corresponds to the sum of the $\rho$. When $N={n\choose k}$, the estimator in  \cref{eqn:gu} is a generalized complete U-statistic and is denoted as $\ru_{n,k,\omega}$.  When $N < {n\choose k}$, these estimators are generalized incomplete U-statistics. 
\end{definition} 
Generalized U-statistics are essentially incomplete U-statistics with potentially extra randomness and where the order of the kernel may grow with $n$.  Random forests, for example, are generalized U-statistics in which the additional randomness $\omega$ determines which features are eligible for splitting at each node in each tree.
In recent work, \cite{peng2019asymptotic} proved that if $\frac{k}{n}({\zeta_k}/{k\zeta_{1,\omega}}-1)\to 0$, then the complete generalized U-statistic $\ru_{n,k,\omega}$ is asymptoticly normal with variance $\frac{k^2}{n}\zeta_{1,\omega}$, where $\zeta_k= \V(s)$ and $\zeta_{1,\omega} =V_1=  \V(\E[s|X_1])$. Let $e^{\omega}_i = {n-1\choose k-1}^{-1} \sum s(X_i, \dots;\omega)$ and let $s_0^{\omega} = {n \choose k}^{-1}\sum s(\dots; \omega)$. Note that each collection of subsamples is paired with an i.i.d. $\omega$. Then the corresponding $\sij_\ru$ for generalized U-statistics is defined as 
\begin{equation}
	\label{eq:psIJomega}
	\sij^{\omega}_\ru = \frac{k^2}{n^2}\sum [e_i^\omega - s_0^\omega]^2.
\end{equation}
The following theorem gives that the same conditions used by \cite{peng2019asymptotic} to establish asymptotic normality are sufficient to establish the consistency of $\sij^\omega_\ru$.  In other words, if the (generalized) U-statistic is nearly linear, then  $\sij ^\omega_\ru/ \V(\ru_{n,k,\omega})$ converges to 1 in probability.

\begin{theorem}
	\label{U-1}
	Let $X_1,\dots, X_n$ be i.i.d.\ from $\rp$ and  $\ru_{n,k,\omega}$ be a generalized complete U-statistic with kernel $s(X_1,\dots,X_{k};\omega)$. Let $\theta=\E[s]$, $\zeta_{1,\omega}=\V(\E[s|X_1])$ and $\zeta_k=\V(s)$.
	If $\frac{k}{n}(\frac{\zeta_k}{k\zeta_{1,\omega}}-1)\to 0 $, then  
	\begin{equation}
		\sij^\omega_\ru /\V(\ru_{n,k, \omega})\xrightarrow[]{p}1.
	\end{equation} 
\end{theorem}

\begin{corollary}
	If the conditions in \cref{U-1} are met, then 
	\begin{equation}
		\begin{aligned}
			\label{eq:corpsij}
			\ru_{n,k,\omega} \pm z_{\alpha/2}\frac{n}{n-k}\sqrt{\sij^\omega_\ru} 
			& =  \ru_{n,k,\omega} \pm z_{\alpha/2}\frac{n}{n-k}\sqrt{\sum \cov^{\omega}_*(s^*,w_i^*)^2}\\
			& =   \ru_{n,k,\omega} \pm z_{\alpha/2}\frac{k}{n-k}\sqrt{\sum(e^\omega_i-s^\omega_0)^2}
		\end{aligned}
	\end{equation} 
	provides an asymptotically valid confidence interval for $\theta$ with confidence level $1-\alpha$. Note that $\cov^\omega_*(s^*, w_i^*)$ can be defined in the same fashion as \cref{eq:psIJomega} to include the extra randomness and equals $\frac{k}{n}(e_i^\omega-s_0^\omega)$.
\end{corollary}

\noindent The $\frac{n}{n-k}$ appearing in \cref{eq:corpsij} is there to correcting for finite sample bias.

Because of their inherent computational burden in calculating ${n\choose k}$ base estimates, the complete forms of these estimators are almost never utilized in practice.  Fortunately, asymptotic normality for incomplete generalized U-statistics was also established in \cite{peng2019asymptotic}. Thus, to establish asymptotically valid confidence intervals for these incomplete counterparts, we need only establish a consistent means of estimating the asymptotic variance of $\ru_{n,k,N,\omega}$. 

Let 
\begin{equation}
	\begin{aligned}
		\label{sij-dagger}
		\widehat{\sij_\ru^\omega} &= \frac{k^2}{n^2}\sum_i [\hat{e}^\omega_i - \hat{s}^\omega_0]^2,
	\end{aligned}
\end{equation}
where 
\begin{equation}
	\hat{s}^\omega_0 = \frac{1}{N}\sum s(...;\omega), \quad \hat{e}^\omega_i = \frac{n}{Nk}\sum s(X_i,...;\omega).
\end{equation}
Here, $\sum s(...;\omega)$ denotes the sum of all kernels that make up the incomplete U-statistic, whereas $\sum s(X_i,...;\omega)$ denotes the sum of all kernels that make up the incomplete U-statistic and include $X_i$ in their respective subsamples. The following theorem shows that we can obtain a consistent estimate of $\zeta_{1,\omega}$ so long as $N$ is large enough to ensure that $\frac{n}{Nk\zeta_{1,\omega}}\to 0$.
Importantly, this means that $N$ need not be on the order of ${n\choose k}$.

\begin{theorem}
	\label{U-2}
	Let $X_1,\dots,X_n$ be i.i.d.\ from $\rp$ and $\widehat{\sij_\ru^\omega}$ be as  \cref{sij-dagger}. Let  $\zeta_{1,\omega}=\V(\E[s|X_1])$ and  $\zeta_k=\V(s)$. Assume that $\E[s]$ and $\zeta_k$ are bounded.
	Then if $\frac{k}{n} (\frac{\zeta_k}{k\zeta_{1,\omega}}-1)  \to  0 $ and 
	$\frac{n}{Nk\zeta_{1,\omega}}\to 0 $, 
	we have
	\begin{equation}
		\widehat{\sij^\omega_\ru} / \frac{k^2}{n}\zeta_{1,\omega}\xrightarrow[]{p}1.
	\end{equation}
\end{theorem}
\begin{remark}
	Consider the case that $ \zeta_k/k\zeta_{1,\omega}\leq c_1$ for and $k\zeta_k\geq c_2$ for some constants $c_1$ and $c_2$. \cref{U-2} states that we need only have $n\gg k$ and $N\gg nk$ in order to ensure that $\widehat{\sij^\omega_\ru}/ \frac{k^2}{n}\zeta_{1,\omega}  \xrightarrow[]{p}1$.
\end{remark}

In the recent literature on random forests discussed above, $\sum_i \widehat{\cov}^2(s^*,w_i^*)$ is often the quantity used to estimate the variance predictions in practice.  The estimator $\widehat{\sij^\omega_\ru}$ has a strong connection to this quantity. Indeed, 
\begin{equation}
	\begin{aligned}
		\frac{k^2}{n^2}\sum_i [\hat{e}_i^\omega - \hat{s}_0^\omega]^2 
		& = \frac{k^2}{n^2}\cdot \frac{n^2}{k^2}\sum [\frac{1}{N}s(...;\omega) w_i^* - \frac{1}{N}\sum s(...;\omega) \frac{k}{n}]^2 \\
		& \approx  \frac{\hat{N}^2}{N^2}  \sum_i \left[ \frac{1}{\hat{N}}\sum s(\dots;\omega)w_i^* - \frac{1}{\hat{N}} \sum s(\dots;\omega)\frac{\hat{N}_i}{\hat{N}}\right]^2\\
		& =  \frac{\hat{N}^2}{N^2} \sum_i \widehat{\cov}^2(s^*, w_i^*)	\approx \sum_i \widehat{\cov}^2(s^*,w_i^*).
	\end{aligned}
\end{equation}
where $\hat{N}_i = \sum w_i$. And since $\hat{N}$ concentrates around $N$ and $\hat{N}_i/\hat{N}$ concentrates around $\E_*[w^*_i] =\frac{k}{n}$, $\widehat{\sij_\ru^\omega} $ should be close to $\sum_i \widehat{\cov}^2(s^*,w_i^*)$.

To our knowledge, \cref{U-2} is the first result that provides a consistency of estimate the variance of predictions generated by subsampled random forests, where the number of trees in the random forest need not to be ${n\choose k}$ and the trees themselves may be randomized.

Recent work has established that $\ru_{n,k,N,\omega}$ has different asymptotic distributions depending on the number of subsamples $N$ that are employed \cite{peng2019asymptotic}. When $N\ll n/k$, $\ru_{n,k,N,\omega} \sim \mathcal{N}(0, \zeta_k/N)$, when $N = O(n/k)$, $\ru_{n,k,N,\omega} \sim \mathcal{N}(0, \frac{k^2}{n}\zeta_{1,\omega} + \frac{\zeta_k}{N})$, and when $N \gg n/k$ and $\mathcal{N}(0, \frac{k^2}{n}\zeta_{1,\omega})$.  Thankfully, regardless of the setting, there are only two variance parameters that may need to be estimated:  $\zeta_k$ and $\zeta_{1,\omega}$.  We can estimate  $\zeta_k$ simply by calculating sample variance of base learners built on non-overlapping subsamples. Based on the above arguments, $\zeta_{1,\omega}$ can be estimated by $\widehat{\sij^\omega_\ru}$. Therefore, it is guaranteed that 
\begin{equation}
	\widehat{\sij^\omega_\ru} \xrightarrow[]{p} \frac{k^2}{n}\zeta_{1,\omega} \quad \text{and}\quad \widehat{\zeta_k}\xrightarrow[]{p} \zeta_k.
\end{equation}

Interestingly, if a random forest is built with $N = O(n)$ decision trees, we cannot estimate the variance of the random forest consistently using only the trees contained in the random forest; this requires $\gg \frac{n}{k\zeta_{1,\omega}}$ decision trees. Note that $k\zeta_{1,\omega}\leq \zeta_k$ and $\zeta_k$ typically tends toward $0$ as $k$ grows. Therefore, the number of trees required for consistent variance estimation is $\gg n$.
These results shed light on the intuition established throughout the machine learning literature that it is always significantly more computationally intensive to estimate the variance of ensembles than to obtain the ensemble itself.

\section{Discussion}	
The work above provides an in-depth examination of the infinitesimal jackknife estimate of variance for resampled statistics.  We provided alternative perspectives and derivations of the estimator, perhaps most notably demonstrating its equivalence to an OLS linear regression of the bootstrap estimates on their respective sampling weights.  Ultimately we derived three alternative estimation methods under the bootstrap regime and demonstrated their equivalence when all bootstrap samples are employed.  
We also gave an in-depth examination of both the Monte Carlo and sampling bias of the IJ estimator in the bootstrap setting and proposed a novel bias-corrected estimator, ultimately providing conditions related to the asymptotic linearity of the statistic under which the IJ estimator is asymptotically unbiased.
In the latter portion of the work, we examine how these preliminary results translate outside the bootstrap setup by looking at setups where the estimator is formed via subsampling without replacement.  Here the statistics admit a familiar U-statistic structure and we derived corresponding results for generalized U-statistics, now a popular tool for analyzing modern supervised learning ensembles like random forests.  Here we also provided a formal motivation for the pseudo IJ often employed in such settings and establish its consistency under similar linearity conditions.  Importantly, for the first time, we further established consistency for the finite sample (incomplete) versions of these estimators.  In particular, this means that one needn't utilize all possible subsamples in order for the empirical $\sij_\ru$ to form a consistent estimator.

Finally, recall from the previous section that in dealing with the U-statistic estimators, we assume that the statistics are approximately linear so that when we write $\V(\ru) =\sum_{j=1}^k {k\choose j}^2{n\choose j}^{-1}V_j$, the variance is dominated by the first order term $k^2/n V_1$ and we propose an estimate of $V_1$ accordingly.  One may wonder whether, if the remaining terms are not negligible, an improved estimate can be provided by also including estimates for  $V_j$ for $j=2,\dots, k$.  While such estimates can be provided, establishing the superiority of the resulting estimator is a far more in depth undertaking that we reserve for future work.  Further detailed discussion is provided in Appendix \ref{app:higherorderIJ}.


\bibliographystyle{Chicago}

\appendix

\section{IJ via OLS Linear Regression}
\label{app:OLSconnection}

In what follows, we outline in detail the connection between OLS linear regression and the infinitesimal jackknife.  In particular, we show how the infinitesimal jackknife estimator of variance of bagged estimates derived recently by Efron \cite{Efron2014} can equivalently be obtained via a straightforward linear regression of the bootstrap estimates on their respective sampling weights.  We begin with some general preliminary results for a general resampling setup and then transition into specific findings for the bootstrap regime.

\begin{remark}
	\cref{app:OLSconnection} should be read as a standalone section.  In particular, because our goal is to cast everything in a familiar regression context, the notation used here differs slightly in some instances from that utilized in the main text and in the remaining appendices.
\end{remark}

Suppose we have a sample $Z_1, ..., Z_n \sim P$ with realized observations $\bm{z}=(z_1, ..., z_n)$ from which we construct an estimator $\hat{\theta} = s(\bm{z})$ for some parameter of interest $\theta$.  Let $\datmat = \left[ Z_1 \cdots Z_n \right]^T$ denote the original data matrix.

Consider a general resampling setup and let $\bm{z}^{*}_{1}, ..., \bm{z}^{*}_{B}$ denote $B$ resamples of the original data that are used to construct the corresponding estimates $\hat{\theta}_1, ..., \hat{\theta}_B$.  Let $w_b = (w_{b,1}, ..., w_{b,n})^T$ denote the associated resampling weights that count the number of times each observation (row) in $\datmat$ appears in each resample.  That is, $w_{b,j} = c$ indicates that the $j^{th}$ sample (row) of $\datmat$ appears exactly $c$ times in the $b^{th}$ resample.  Denote the average across these resampled estimates by
\[
\tilth = \frac{1}{B} \sum_{b=1}^{B} \hth_b
\]

\noindent so that $\tilth$ corresponds to the standard bagged estimate of $\theta$ whenever bootstrapping is the particular kind of resampling employed.  Our primary goal in the following subsection is to derive a closed form estimate for $\var(\tilth)$ in the bootstrap regime.  We begin by deriving some more general preliminary results.

First note that by the law of total variance, we can write
\begin{align*}
	\var(\tilth) &= \mathbb{E} \left[ \var \left(\tilth \big| \datmat \right) \right] + \var\left[ \mathbb{E} \left(\tilth \big| \datmat \right) \right] \\
	&\stackrel{B}{\approx} \var\left[ \mathbb{E} \left(\tilth \big| \datmat \right) \right]
\end{align*}
since $\left(\tilth | \datmat \right) \rightarrow 0$ as $B \rightarrow \infty$.  Further, we have that
\begin{align*}
	\mathbb{E} \left(\tilth \big| \datmat \right) &= \frac{1}{B} \mathbb{E} \left[ 1_{B}^{T} \left(\hth_1 \cdots \hth_B\right)^T \; \middle| \; \datmat \right] \\
	&= \frac{1}{B} 1_{B}^{T} \left[ \mathbb{E} \left\{\left(\hth_1 \cdots \hth_B\right)^T \; \middle| \; \datmat \right\} \right] \\
	&= \frac{1}{B} \sum_{b=1}^{B} \mathbb{E} \left( \hat{\theta}_b \middle| \datmat \right)
\end{align*}

\noindent so that for $\gamma_b = \mathbb{E} \left(\hat{\theta}_b \middle| \datmat \right)$, we can write 
\begin{align}
	\var\left(\tilth \right) &\stackrel{B}{\approx} \var\left( \frac{1}{B} \sum_{b=1}^{B} \mathbb{E} \left(\hat{\theta}_b | \datmat \right) \right) \nonumber \\
	&= \frac{1}{B^2} \sum_{b=1}^{B} \var(\gamma_b) + \frac{1}{B^2} \sum_{b>b^{'}} \cov(\gamma_b,\gamma_{b^{'}}) \nonumber \\
	&\stackrel{B}{\approx} \frac{1}{B^2} \sum_{b>b^{'}} \cov(\gamma_b,\gamma_{b^{'}}).
\end{align}

\noindent Note also that so long as the resampling weights are identically distributed,
\begin{align*}
	\cov\left( \gamma_b, \gamma_{b^{'}}\right) = \rho_{b,b^{'}} \sqrt{\var(\gamma_b) \var(\gamma_{b^{'}})} = \var(\gamma_b)
\end{align*}

\noindent where $\rho_{b,b^{'}}$ denotes the correlation between $\gamma_b$ and $\gamma_{b^{'}}$.  Thus,
\begin{equation}
	\label{eqn:vartilvargam}
	\var(\tilth) \approx \frac{B (B-1)}{B^2} \var(\gamma_b) \stackrel{B}{\approx} \var(\gamma_b).
\end{equation}

We close this Section with a final key observation:  in this setup, conditional on $\datmat$, for each resample $b = 1, ..., B$ we can write
\[
\hat{\theta}_b = g(w_b)
\]

\noindent for some (unknown) function $g$.  Thus, in order to investigate the properties of the resampled estimates, we need only understand how $g$ depends on $w_b$.

\subsection*{The Bootstrap Setting}

We now narrow our focus to the bootstrap regime where $B$ equally-weighted resamples of size $n$ are independently taken from the rows of $\datmat$ with replacement so that each weight vector $w_b$ is thus distributed as $Multinomial(n; \frac{1}{n}, ..., \frac{1}{n})$.  Now note that conditional on $\datmat$, for each resample $b = 1, ..., B$ we can write
\[
\hat{\theta}_b = g(w_b)
\]
so that $g(1_n)$ gives the estimate based on all original observations $\hat{\theta}$.  Importantly, this means that in order to investigate the properties of the bootstrap estimates, we need only understand how $g$ depends on the weights $w_b$.

For each bootstrap replicate $b$, we can write
\[
\hat{\theta}_b - \hat{\theta} = g(w_b) - g(1_n).
\]

Now, if we assume that $g$ is differentiable, then a first-order Taylor approximation to $\hat{\theta}_b - \hat{\theta}$ is given by 
\[
g(1_n)^{T}(w_b - 1_n).
\]

\noindent Absent this differentiability assumption, we could alternatively consider modeling the underlying relationship $g$ linearly via
\begin{equation}
	\label{eqn:linreg}
	\hat{\theta}_b - \hat{\theta} = \beta^{T} (w_b - 1_n) + \epsilon_b \quad \quad \text{ for } b = 1, ..., B.
\end{equation}

\noindent Taking this approach, the ordinary least squares estimate for $\beta$ is given by
\begin{align}
	\hat{\beta}_{\text{OLS}}  = \hbh &= \argmin \sum_{b=1}^{B} \left( \hat{\theta}_b - \hat{\theta} -  \beta^{T} (w_b - 1_n) \right)^2 \nonumber \label{eqn:betahat} \\
	&= \left( \frac{1}{B-1} X^T X \right)^{-1} \left( \frac{1}{B-1} X^T Y \right)
\end{align}

\noindent where $Y = (\hat{\theta}_1 - \hat{\theta}, ..., \hat{\theta}_B - \hat{\theta})^T$ and $X = ((w_1 - 1_n)^T, ..., (w_B - 1_n)^T)^T$ correspond to the (centered) bootstrap estimates and weights, respectively.  

Recall from (\ref{eqn:vartilvargam}) that 
\[
\var(\tilth) \approx \var(\gamma_b).
\]

\noindent where $\gamma_b = \mathbb{E} \left(\hat{\theta}_b \middle| \datmat \right)$.  Now, $\var(\gamma_b)$ can be  estimated by 
\begin{align*}
	&\frac{1}{B-1} \sum_{b=1}^{B} \left( \gamma_b - \mathbb{E}(\gamma_b) \right)^2 \stackrel{B}{\approx} \frac{1}{B-1} \sum_{b=1}^{B} \left( \gamma_b - \bar{\gamma}_b \right)^2
\end{align*}
and from our linear regression model in (\ref{eqn:linreg}),
\[
\hat{\gamma_b} = \hat{\theta} + \hat{\beta}^{T} (w_b - 1).
\]

Further, note that since $w_1, ..., w_B \iid Multinomial(n; \frac{1}{n}, ..., \frac{1}{n})$, each $w_{b,i} \sim Binom(n,\frac{1}{n})$ so that $\mathbb{E}(w_{b,i}) = 1$ for all $b = 1, ..., B$ and all $i = 1, ..., n$.  Thus, estimating the variance of $\tilth$ with the sample variance of $\{\hat{\gamma}_1, ..., \hat{\gamma}_B\}$, we have 
\begin{align*}
	\widehat{\var}(\tilth) &\approx \frac{1}{B-1} \sum_{b=1}^{B} \left( \hat{\beta}^{T} (w_b - 1) \right)^2 \\
	&= \hat{\beta}^{T} \left[ \frac{1}{B-1} \sum_{b=1}^{B} (w_b - 1) (w_b - 1)^T  \right] \hat{\beta}.
\end{align*}

\noindent Looking at the middle term, we have
\begin{align*}
	\left[ \frac{1}{B-1} \sum_{b=1}^{B} (w_b - 1) (w_b - 1)^T  \right] &\stackrel{B}{\approx} \mathbb{E} \left( (w_1 - 1) (w_1 - 1)^T \right) \\
	&= \cov(w_1 - 1) \\
	&= I_n - \frac{1}{n}1_{n \times n} \\
	&\stackrel{n}{\approx} I_n
\end{align*}

\noindent and thus, 
\begin{equation}
	\label{eqn:vartilthbeta}
	\widehat{\var}(\tilth) \approx \hat{\beta}^T I_n \hat{\beta} = \sum_{j=1}^{n} \hat{\beta}_{j}^{2}.
\end{equation}

\noindent Thus, in order to produce our estimate for $\var(\tilth)$, it remains only to work out the solution to $\hat{\beta}$ given in (\ref{eqn:betahat}).  

Let's begin by considering the expectation of the inverse of the first term in (\ref{eqn:betahat}).  Observe that for $i \neq j$
\begin{align*}
	\left[ \mathbb{E} \left( \frac{1}{B-1} X^TX \right) \right]_{i,j} &= \frac{1}{B-1} \mathbb{E} \left( \sum_{k=1}^{B} (w_{k,i} - 1)(w_{k,j} - 1) \right) \\
	&= \frac{1}{B-1} \sum_{k=1}^{B} \mathbb{E} (w_{k,i} - 1)(w_{k,j} - 1) \\
	&= \frac{1}{B-1} \sum_{k=1}^{B} \cov(w_{k,i}, w_{k,j}) \\
	&= \frac{B}{B-1} \cov(w_{1,i}, w_{1,j}) \\
	&= \frac{B}{B-1} \left( -n \right)  \left( \frac{1}{n} \right)  \left( \frac{1}{n} \right) \\
	&= \left( \frac{B}{B-1} \right) \frac{-1}{n} 
\end{align*}

\noindent where the third equality comes from the fact that each $w_{b,i} \sim Binom(n,\frac{1}{n})$ and the fourth and fifth equalities follow from $w_1, ..., w_B \iid Multinomial(n; \frac{1}{n}, ..., \frac{1}{n})$.  For the diagonal elements ($i=j$), the covariance terms above become variance terms so that 
\[
\frac{B}{B-1} \cov(w_{1,i}, w_{1,j}) = \frac{B}{B-1} \var(w_{1,i}) = \left( \frac{B}{B-1} \right) \frac{n-1}{n}
\]

\noindent and thus, in matrix form, we have 
\[
\mathbb{E}\left( \frac{1}{B-1} X^T X \right) = -\frac{1}{n} 1_{n \times n} + I_n.
\]

\noindent Finally, note that

\begin{equation}
	\label{eqn:xtxI}
	\left(\frac{1}{B-1} X^T X \right) \stackrel{B}{\approx} \mathbb{E} \left( \frac{1}{B-1} X^T X \right) \stackrel{n}{\approx} I_n
\end{equation}

\noindent and thus
\[
\hbh = \left( \frac{1}{B-1} X^T X \right)^{-1} \left( \frac{1}{B-1} X^T Y \right) \approx I_{n}^{-1} \left( \frac{1}{B-1} X^T Y \right) = \frac{1}{B-1} X^T Y.
\]

\noindent Now, 
\begin{align*}
	\frac{1}{B-1} X^T Y  &= \frac{1}{B-1} \bigg( (w_1 - 1_n)^T, ..., (w_B - 1_n)^T \bigg)    \bigg( (\hat{\theta}_1 - \hat{\theta}), ..., (\hat{\theta}_B - \hat{\theta}) \bigg)^T \\
	&= \bigg( \frac{1}{B-1} \sum_{b=1}^{B} (w_{b,1} - 1)(\hat{\theta}_b - \hat{\theta}), \; ... \;, \frac{1}{B-1} \sum_{b=1}^{B} (w_{b,n} - 1)(\hat{\theta}_b - \hat{\theta}) \bigg)^T
\end{align*}

\noindent so that elementwise, 

\[
\hbh_j = \frac{1}{B-1} \sum_{b=1}^{B} (w_{b,j} - 1) (\hth_b - \hth)
\]

\noindent and since $\mathbb{E}(w_{b,j}) = 1$, $\hbh_j$ is effectively the sample covariance of $(w_{1,j}, ..., w_{B,j})$ and $(\hth_1, ..., \hth_B)$.  Denoting this by sample covariance by $\widehat{\cov}_{j}$ and putting this together with (\ref{eqn:vartilthbeta}), we have 
\begin{equation}
	\label{eqn:vartilthfinal}
	\widehat{\var}(\tilth) \approx \sum_{j=1}^{n} \hat{\beta}_{j}^{2} = \sum_{j=1}^{n} \widehat{\cov}_{j}^{2}
\end{equation}

\noindent which coincides exactly with the infinitesimal jackknife variance estimate given by Efron in \cite{Efron2014}.

\section{Proofs and Calculations for $\ij_\rb$ (IJ for Bootstrap)}
\label{app:bootstrap}

\noindent \textbf{Proof of \cref{3rivers}:}

\begin{enumerate}
	\item By definition, 
	\begin{equation*}
		\begin{aligned}
			\E_*[s^* w_j^*] & = \sum_{w_1^*+\dots+w_n^* = n}p(w_1^*,\dots, w_n^*) s(X_1^*,\dots, X_n^*) w_j^* \\
			& = \sum_{\substack{w_j^*\geq 1 \\w_1^*+\dots+w_n^* = n}}\frac{(n-1)!}{w_1^*\dots ((w_j^*-1)!)\cdots (w_n^*)!}\frac{1}{n^{n-1}} s(X_1^*,\dots, X_n^*)\\
			& = \E_*[s(X_1^*,\dots, X_n^*)|X_1^*=X_j] \\
			& = e_j.
		\end{aligned}
	\end{equation*}
	
	\item Conditional on the data, knowing $X_1^*,\dots, X_n^*$ is equivalent to knowing $w_1^*,\dots, w_n^*$. Therefore, $l^*$ can be also viewed as the  projection of $s^*$ onto the linear space spanned by $X_1^*,\dots, X_n^*$. Then we have
	\begin{equation*}
		\begin{aligned}
			l^* &= \sum_i ( \E_*[s^*|X_i^*] - \E_*[s^*]) \\
			&= \sum_i \sum_j (\E_*[s^*|X_i^*=X_j] - \E_*[s^*]) \bone_{\{X_i^*=X_j\}} \\
			& = \sum_i \sum_j(e_j-s_0)\bone_{\{X_i^*=X_j\}} \\
			& = \sum_j \sum_i (e_j-s_0)\bone_{\{X_i^*=X_j\}} \\
			& = \sum_j w_j^* (e_j-s_0)
		\end{aligned}
	\end{equation*}
	as desired.
	
	\item By 1 above, $\mathrm{IJ_B} = \sum_j \cov^2_*(s^*, w_j^*) = \sum_j (\E_*[s^*w_j^*] -\E_*[s^*]\E_*[w_j^*] )^2= \sum_j (e_j-s_0)^2 $. By 2, we have $\V_*(l^*) = \V_*(\sum_j w_j^*(e_j-s_0))  =\sum_j (e_j-s_0)^2$.
	Thus, $\V_*(l^*) =  \mathrm{JK_B} = \mathrm{IJ_B}$.
\end{enumerate} \hfill $\blacksquare$ \\

\noindent\textbf{Calculations  for sample maximum:} Consider $s=\max\{X_1,...,X_n\}$, where $X_1,\dots, X_n$ are uniformly distributed in $[0,1]$. For the order statistics $X_{(1)}\leq X_{(2)}\leq \dots \leq X_{(n)}$, we have 
\begin{equation}
	\label{order-ij}
	\\cov(X_{(i)},X_{(j)}) = \frac{i(n-j+1)}{(n+1)^2(n+2)} \quad\text{and}\quad  \E[X_{(i)}X_{(j)}] = \frac{i(j+1)}{(n+1)(n+2)}.
\end{equation}
Note that 
\begin{equation*}
	\E_*[s^*] = s_0 =  \sum_{i=1}^n X_{(i)} p^n_i,
\end{equation*}
where $p^n_i = q^n_i - q^n_{i-1}$ and $q^n_i=\left(\frac{i}{n}\right)^n$ for $i= 1, \dots, n$.  
Thus, 
\begin{equation}
	\V(\E_*[s^*]) = v^T \mathrm{A} v
\end{equation}
where $	\A = cov(\U) = \left[\frac{i(n+1-j)}{(n+1)^2(n+2)}\right]_{ij}$ 
and $v = (p_1^n, \cdots, p_n^n)$. Let 
\begin{equation*}
	\begin{aligned}
		\tilde{e}_i
		& =  \sum_{j=I+1}^n X_{(j)} p^{n-1}_j + X_{(i)} q^{n-1}_i ,\quad \text{where } q_i^{n-1} = \left(\frac{i}{n}\right)^{n-1}.
	\end{aligned}
\end{equation*}
We have $\V_*(l^*)=  \sum_{i=1}^n (e_i-s_0)^2  = \sum_{i=1}^n (\tilde{e}_i-s_0)^2$.
Thus, 
\begin{equation}
	\begin{aligned}
		\E[\V_*(l^*)] 
		& = \sum(v_i-v)^T\B (v_i-v)
	\end{aligned}
\end{equation}
where $\B = \E[\U\U^T] = 
\begin{bmatrix}
	\frac{i(j+1)}{(n+1)(n+2)}
\end{bmatrix}_{ij}$ and $v_i = (\cdots,  0, \cdots,q_i^{n-1},\cdots,p_j^{n-1}, \cdots)$.
Next, we have 
\begin{equation}
	\begin{aligned}
		v^T\mathrm{A}v =
		& = \frac{1}{(n+1)^2(n+2)}  \sum_i \sum_j p_i^np_j^n {i(n+1-j)}\\
		& = \frac{1}{(n+1)^2(n+2)}\left(\sum_ii\cdot p_i^n\right) \left(\sum_j (n-j+1)p_j^n\right)\\
		& =  \frac{1}{(n+1)^2(n+2)} \left(\sum_i i\cdot p_i^n\right) \left(n+1-\sum_i i\cdot p_i^n\right)\\
		& =  \frac{1}{(n+1)^2(n+2)} (n - \sum_j (\frac{j-1}{n})^n) )(1 + \sum_j (\frac{j-1}{n})^n))
	\end{aligned}
\end{equation}
by the fact that 
\begin{equation}
	\begin{aligned}
		\sum_{i=1}^n i\cdot p_i^n 
		= \sum_{i=1}^n iq^n_i  - \sum_{i=0}^{n-1} (i+1)q_{i}^n
		& = n - \sum_{i=0}^{n-1} q_{i}^n.
	\end{aligned}
\end{equation}
Now, let $\mathbf{e}_n= [1,2,\cdots, n]^T$. Then 
\begin{equation}
	\begin{aligned}
		(v_i-v)^T\B	(v_i-v)& = \frac{ 	(v_i-v)^T
			\mathbf{e}_n
			\cdot
			(\mathbf{e}_n^T + 1_n^T)\mathrm{V}}{(n+1)(n+2)} \\
		&   = \frac{	(v_i-v)
			\mathbf{e}_n
			\cdot
			\mathbf{e}_n^T	(v_i-v)}{(n+1)(n+2)} \\
		& = \frac{1}{(n+1)(n+2)}  \sum_i \left[(n-\sum_{j=I}^{n-1}(\frac{j}{n})^{n-1}) - (n-\sum_{j=0}^{n-1} (\frac{j}{n})^n)\right]^2 \\
		& = \frac{1}{(n+1)(n+2)}  \sum_i \left[\sum_{j=1}^{n} (\frac{j-1}{n})^n - \sum_{j=i+1}^{n}(\frac{j-1}{n})^{n-1}\right]^2.
	\end{aligned}
\end{equation}
In summary, we have 
\begin{equation}
	\begin{aligned}
		\frac{\E[\V_*(l^*)]}{\V(\E_*[s^*])} 
		& =  \frac{(n+1)\sum_i [ \sum_{j=1}^{n} (\frac{j-1}{n})^n - \sum_{j=i+1}^{n}(\frac{j-1}{n})^{n-1}]^2}{(n-\sum_j  (\frac{j-1}{n})^n) (1+ \sum_j (\frac{j-1}{n})^n)}\\
		& \to c  \in [0.24,0.25]\quad \text{as } n \to \infty. 
	\end{aligned}
\end{equation}

\noindent\textbf{Proof of \cref{BOOT-2}: } Note that  $\E[\V_*(l^*)] = (n-1)\E[e_1^2 - e_1e_2]$ and  $\V(\E_*[s^*]) =  \frac{1}{n}\V(e_1) + \frac{n-1}{n}{\\cov}(e_1,e_2)$. Let $\rho = {\\cov}(e_1, e_2) /\V(e_1) $.  We have 
\begin{equation}
	\begin{aligned}
		\E[\V_*(l^*)] / \V(\E_*[s^*]) & =\frac{(n-1)(1-\rho)}{{1}/{n}+ {(n-1)}/{n}\cdot \rho} \\
		& = n \frac{1-\rho}{1/(n-1) + \rho}.
	\end{aligned}
\end{equation}
Now, consider the Taylor expansion of $f(\rho) =  \frac{n(1-\rho)}{1/(n-1) + \rho}$.  We have
\begin{equation}
	\begin{aligned}
		f(\rho) 
		& = n \left(-1 + \frac{1}{1- (n-1)/n\cdot(1-\rho) }\right)\\
		& = n(-1+\frac{1}{1-r}) \\ 
		& = n \left(-1 + 1  + r + r^2 + r^3 +\dots \right) \quad (|r|<1)\\ 
		&  =  n(r + r^2 + r^3 +  \dots) \quad (|r<1|),
	\end{aligned}
\end{equation}
where $r = (n-1)/n \cdot (1-\rho)$. 
Therefore, only if $r= \frac{1}{n} + o(\frac{1}{n})$, then $f(\rho(r)) \to 1 $ as $n\to \infty$. In particular, 
\begin{equation}
	\begin{aligned}
		r = \frac{1}{n} + o\left( \frac{1}{n} \right) 
		& \iff 1-\rho = \frac{1}{n}  + o \left(\frac{1}{n} \right).
	\end{aligned}
\end{equation}
Hence  $\lim_{n\to \infty}f(\rho) =  1$ if and only if $\lim_{n\to \infty}n(1-\rho) = 1$.  
Thus,  $\ij_\rb$ is an asymptotically unbiased estimator of $\V(\E_*[s^*])$ if and only if $1-\rho = {1}/{n} + o(1/n)$.  \hfill $\blacksquare$ \\


\section{Proofs and Calculations for $\ij_\ru$ and $\sij_\ru$  (IJ for U-statistics)}
\label{app:ustatistics}

\noindent \textbf{Proof of \cref{U-ij-1}:}
When subsampling without replacement, according to the weight of each sample, the probability of $(x_1,\dots, x_k)$ being selected is 
\begin{equation}
	\begin{cases}
		\sum_{i_1,\dots, i_k} \frac{\P_n(x_{i_1})}{1}\times \frac{\P_n(x_{i_2})}{1-\P_n(x_{i_1})}\times \cdots \times \frac{\P_n(x_{i_k})}{1-\sum_{j=1}^{k-1}\P_n(x_{i_j})}, & x_1,\dots x_k \in \D_n \text{ and are distinct}. \\
		0, & \text{otherwise}.
	\end{cases}
\end{equation}
Note that any subsampling with a general re-weighting scheme can be derived similarly. Consider $f((1-\epsilon)\P_n + \epsilon \delta_{X_i} )$ and let $\delta= 1-\epsilon$. We first provide the probability of obtaining $(x_1,\dots, x_k)$. On one hand, if $X_i\not \in (x_1,x_2,\dots, x_k)$, then 
\begin{equation}
	p(x_1,x_2,\dots, x_k)= p_{0} = \left[\frac{\delta}{n}\cdot \frac{\delta}{(n-\delta)}\cdots\frac{\delta}{(n-(k-1)\delta)}\right]\times k!.
\end{equation}
On the other hand, if $X_i \in (x_1,\dots, x_k)$, then $p(x_1,x_2,\dots, x_k) = p_1=\sum_{i=0}^{k-1} q_i$, where 
\begin{equation}
	\begin{aligned}
		q_0 &= \left[\frac{(n-(n-1)\delta)}{n}\cdot \frac{1}{n-1}\cdots \frac{1}{n-k+1} \right]\times (k-1)! \\
		q_1 & = \left[\frac{\delta}{n}\cdot \frac{n-(n-1)\delta}{n-\delta}\cdot \frac{1}{n-2}\cdots \frac{1}{n-k+1}\right]\times (k-1)!\\
		\vdots \\
		q_{k-1} & = \left[\frac{\delta}{n}\frac{\delta}{n-\delta}\cdots \frac{\delta}{n-(k-2)\delta}\cdot \frac{n-(n-1)\delta}{n-(k-1)\delta}\right] \times (k-1)!.
	\end{aligned}
\end{equation}
Thus, 
\begin{equation*}
	f((1-\epsilon)\P_n + \epsilon \delta_{X_i} ) = \sum_{(n,k)} s(X_{i_1}, \dots, X_{i_k}) (p_0 \one_{i \not \in \{i_1,\dots, i_k\}} + p_1\one_{i\in \{i_1,\dots, i_k\}}),
\end{equation*}
where the sum is taken over all ${n\choose k}$ of subsamples of size $k$.
\[
\frac{1}{p}p_0'(\delta)|_{\delta=1}  =  -\left[\frac{0}{n}+\frac{1}{n-1}+\cdots + \frac{k-1}{n-(k-1)}\right]-k
\]
and
\[
\begin{aligned}
	\frac{1}{p}p_1' & = \frac{1}{p}\sum_{j=0}^{k-1}q_j'|_{\delta=1} \\
	& = \frac{1}{k}\sum_{j=0}^{k-1}\left[(n-j-1) - \left[\frac{0}{n}+\frac{1}{n-1}+ \frac{2}{n-2}+\cdots \frac{j}{n-j}\right]\right] \\
	& =-\frac{1}{k}\left[ \frac{0\cdot k}{n}+ \frac{1\cdot (k-1)}{n-1}+ \cdots + \frac{(k-1)\cdot 1}{n-(k-1)}\right] -\frac{k+1}{2} +n.
\end{aligned}
\]
Putting all  together,  we have 
\begin{equation}
	\begin{aligned}
		\rd_i & =\lim_{\delta\to 1}\frac{f(\delta\P_n +(1-\delta)\delta_{X_i})-f(\P_n)}{1-\delta}\\
		& = \sum_{(n,k)}(p'_0\one_{w_i^*=0} + p_1'\one_{w_i^*=1})s(X_{i_1},\dots, X_{i_k}) \\
		& =  \sum_{(n,k)}p\left[\frac{p'_0}{p}+(\frac{p_1'}{p}-\frac{p_0}{p}'){w_i^*}\right]s(X_{i_1},\dots, X_{i_k}) \\
		& = \frac{k}{n}(\frac{p'_1}{p}-\frac{p'_0}{p})e_i + \frac{p'_0}{p}s_0,
	\end{aligned}
\end{equation}
where $p ={n\choose k}^{-1}$, $e_i = \E_*[s^*|X_1^*=X_i]$ and $s_0 = \E_*[s^*]$. And $*$ refers to the procedure of subsampling without replacement. Then the infinitesimal jackknife estimate for U-statistic is 
\begin{equation}
	\label{IJ-U}
	\begin{aligned}
		\ij_\ru&= \frac{1}{n^2}\sum _{j=1}^n \left[\frac{k}{n}(\frac{p'_1}{p}-\frac{p'_0}{p})e_j + \frac{p'_0}{p}s_0\right]^2\\
		& = \frac{k^2}{n^2} \sum_{j=1}^n \left[\frac{p'_1 -p'_0}{np}e_j + \frac{p'_0}{kp}  s_0\right]^2\\
		& = \frac{k^2}{n^2}\sum_{j=1}^n [\alpha e_j - \beta s_0]^2
	\end{aligned}
\end{equation}
where
\begin{equation}
	\label{alpha}
	\begin{aligned}
		\alpha &= (p_1'-p_0')/(np) = 1 + \frac{1}{n}\left\{\frac{k-1}{2} - \frac{1}{k}\sum_{j=0}^{k-1}\frac{j^2}{(n-j)}\right\}\\
	\end{aligned},
\end{equation}
and 
\begin{equation}
	\label{beta}
	\begin{aligned}
		\beta &=-p_0'/(kp) = 1 + \frac{1}{k}\sum_{j=0}^{k-1}\frac{j}{n-j}.
	\end{aligned}
\end{equation}  \hfill $\blacksquare$ \\

\noindent \textbf{Proof of \cref{U-ij-2}:}
By definition, 
\begin{equation}
	\begin{aligned} 
		(\alpha e_1- \beta s_0) 
		& =
		-\beta {n\choose k}^{-1}\sum s(X_{i_1},\dots, X_{i_k}; \not \exists 1) \\
		& ~~~~+ \left( \frac{\alpha \cdot (k-1)!}{(n-1)\dots (n-k+1)} -  \frac{\beta \cdot (k-1)!k}{n\cdots (n-k+1)}\right)\sum s(X_{i_1},\dots, X_{i_k}; \exists 1)
		\\
		& = 
		-(1-\frac{k}{n})\beta {n-1\choose k}^{-1}\sum s(X_{i_1},\dots, X_{i_k};\not \exists 1) \\
		& ~~~~+ (\alpha-\frac{k}{n}\beta ){n-1 \choose k-1}^{-1}\sum s(X_{i_1},\dots, X_{i_k}; \exists 1).
	\end{aligned}
\end{equation}
Note that according to H-decomposition, $$s(x_1,\dots, x_k)=\E[s] +\sum_{j=1}^k \sum_{i_1,\dots, i_j\in \{1,\dots ,j\}} s^{j}(x_{i_1},\dots, x_{i_j}).$$ Then 
\begin{equation}
	\begin{aligned}
		(\alpha e_1- \beta s_0) 
		& = (\alpha-\beta)\theta - (1-\frac{k}{n})\beta \sum_{j=1}^k{k\choose j}{n-1\choose j}^{-1}\sum s^{j}(X_{i_1},\dots, X_{i_j}; \not \exists 1)  \\
		& ~~~~+(\alpha -\frac{k}{n}\beta)  \sum_{j=1}^{k-1}{k-1\choose j}{n-1\choose j}^{-1}\sum s^{j}(X_{i_1},\dots, X_{i_j}; \not\exists 1) \\
		&  ~~~~+ (\alpha -\frac{k}{n}\beta) \sum_{j=1}^k {k-1\choose j-1}{n-1\choose j-1}^{-1} \sum s^{j}(X_{i_1},\dots,X_{i_j}; \exists 1)
		\\
		& := (\alpha-\beta)\theta + A_n+B_n,
	\end{aligned}
\end{equation}
where 
\begin{equation*}
	\begin{aligned}
		A_n 
		& =    \sum_{j=1}^{k}\left[(\alpha -\frac{k}{n}\beta) {k-1\choose j}{n-1\choose j}^{-1} - (1-\frac{k}{n})\beta {k\choose j}{n-1\choose j}^{-1} \right]\sum s^{j}(X_{i_1},\dots, X_{i_j}; \not \exists 1)
	\end{aligned}
\end{equation*}
and 
\begin{equation*}
	\begin{aligned}
		B_n 
		& =(\alpha - \frac{k}{n}\beta) \sum_{j=1}^k {k-1\choose j-1}{n-1\choose j-1}^{-1}\sum s^{(j)}(X_{i_1},\dots,X_{i_j} ;\exists 1).
	\end{aligned}
\end{equation*}
Thus, we have 
\begin{equation}
	\begin{aligned}
		\E[A^2_n] 
		& = \sum_{j=1}^k  \left[(\frac{k}{j}-1)\alpha + (\frac{k}{n}-\frac{k}{j})\beta\right]^2{k-1\choose j-1}^2 {n-1\choose j}^{-1}V_j \\
		\E[B^2_n] &=(\alpha - \frac{k}{n}\beta)^2 \sum_{j=1}^k {k-1\choose j-1}^2 {n-1\choose j-1}^{-1}V_j,
	\end{aligned}
\end{equation}
where $V_j = \V(s^j)$. Since $A_n$ and $B_n$ are  uncorrelated and have mean zero, we have 
\begin{equation*}
	\begin{aligned}
		&~~~~ \E\left[ (\alpha e_1-\beta s_0)^2\right]  \\
		& = \E[A^2_n] + \E[B^2_n] + (\alpha-\beta)^2\theta^2 \\
		& =  (\alpha-\beta)^2\theta^2 +  \sum_{j=1}^{k}{k-1\choose j-1}^2 \Lambda(j) V_j,
	\end{aligned}
\end{equation*}
where $\Lambda(j) =  \left[   \left[(\frac{k}{j}-1)\alpha + (\frac{k}{n}-\frac{k}{j})\beta\right]^2{n-1\choose j}^{-1} + (\alpha -\frac{k}{n}\beta)^2{n-1\choose j-1} ^{-1} \right]$, for $j = 1, \cdots, k$.  Therefore,
\begin{equation}
	\begin{aligned}
		\E[\ij_\ru] &= \frac{k^2}{n^2}\sum \E[(\alpha e_j-\beta s_0)^2] \\
		& = \frac{k^2}{n}\sum_{j=1}^k {k-1\choose j-1}^2 \Lambda(j)V_j + \frac{k^2}{n}(\alpha-\beta)^2\theta^2.
	\end{aligned}
\end{equation}
Recall that
\begin{equation}
	\V(\ru) = \sum_{j=1}^k{k\choose j}^2{n\choose j}^{-1}V_j.
\end{equation}
We consider the ratio of the coefficient of $V_j$ in $\E[\mathrm{IJ}_\ru]$ and that in $\V(\ru)$ and obtain 
\begin{equation}
	\begin{aligned}
		r_j & = \frac{k^2}{n} \Lambda(j) {k-1\choose j-1}^2 {k\choose j}^{-2}{n\choose j} \\
		& = \frac{k^2}{n}  \frac{j^2}{k^2} \Lambda(j){n\choose j}\\
		& = \frac{(n-k)^2}{n^2}\left[\frac{j}{1-j/n} \alpha^2\right]+\frac{k^2}{n} (\alpha-\beta)^2 \\
	\end{aligned}
\end{equation}
for  $j=1,\dots, k$. \hfill $\blacksquare$ \\

\noindent \textbf{Proof of \cref{U-sij-pre}:}
By definition, 
\begin{equation}
	\begin{aligned}
		\cov_*(s^*, w_j^*) & = \sum_{w_1^*+\dots+w_n^* = k}p(w_1^*,\dots, w_n^*) [s^*-s_0]w_j^* \\
		& =  \sum_{w_1^*+\dots+w_n^* = k}p(w_1^*,\dots, w_n^*) s^* w_j^*  - \frac{k}{n}s_0 \\
		& = \frac{k}{n}\sum_{w_j^*=1, w_1^*+\dots+w_n^* = k}  \frac{(k-1)!}{(n-1)\cdots (n-k+1)} s^* - \frac{k}{n}s_0 \\
		& = \frac{k}{n}\left[\E_*[s(X_1^*,\dots, X_k^*)|X_1^*=X_j] - s_0\right]\\
		& = \frac{k}{n}(e_j-s_0).
	\end{aligned}
\end{equation}
It follows that $\sij = \sum \cov_*^2(s^*, w_j^*) =  \frac{k^2}{n^2}\sum (e_j-s_0)^2.$
\hfill $\blacksquare$ \\

To proof and \cref{U-1} and \cref{U-2}, we establish the following lemma. 
\begin{lemma}
	\label{gold}
	Suppose that $\sum X_i^2\xrightarrow[]{p}1$, $\sum \E[X_i^2] \to 1$, and $\sum_{i=1}^n \E[Y^2_i] \to 0$, then 
	\begin{equation}
		\sum  [X_i+Y_i]^2 \xrightarrow[]{p} 1 \quad \text{and} \quad \E[\sum (X_i+Y_i)^2] \to 1.
	\end{equation}
\end{lemma}
\begin{proof}
	Note that 
	\begin{equation}
		\begin{aligned}
			\sum (X_i+Y_i)^2 =\sum X_i^2 + \sum Y_i^2 + 2\sum X_iY_i.
		\end{aligned}
	\end{equation}
	Since $\sum \E[Y_i^2]\to 0$, we have $\sum Y_i^2 \xrightarrow[]{l_1} 0$, which implies  that  $\sum Y_i^2 \xrightarrow[]{p}0$. On the other hand, by Cauchy–Schwarz inequality, we have 
	\begin{equation}
		\begin{aligned}
			\E[|\sum_iX_iY_i|] &\leq \sum \sqrt{\E[X_i^2]}\sqrt{\E[Y_i^2]} \\
			& \leq \sqrt{\sum \E[X_i^2]} \sqrt{\sum \E[Y_i^2]}\\
			& \to 1\times  0 \\
			& = 0.
		\end{aligned}
	\end{equation} 
	Thus, $\sum X_iY_i\xrightarrow[]{l_1}0$, which implies that $\sum X_iY_i\xrightarrow[]{p}0$. Therefore, $\sum (X_i+Y_i)^2 \xrightarrow[]{p}0$ by Slutsky's lemma. Furthermore, since $\E[\sum X_iY_i] \to 0$, we have $\E[\sum (X_i+Y_i)^2] \to 1$.
\end{proof}

\noindent \textbf{Proof of \cref{U-1}:}
For simplicity, we first ignore the extra randomness $\omega$. According to the H-decomposition of $s(x_1,\dots, x_k)$, we have 
\begin{equation}
	\label{decom}
	\begin{aligned}
		&~~~~\frac{1}{n}\sum [e_i-s_0]^2 \\
		&=  \frac{1}{n}  \frac{(n-k)^2}{n^2}\sum_{i=1}^n \left[\sum_{j=1}^k -{k-1\choose j-1}{n-1\choose j}^{-1} \sum s^{j}(X_{i_1},\dots, X_{i_j};\not \exists i) \right.\\
		& \left.~~~~ + {k-1\choose j-1}{n-1\choose j-1}^{-1}\sum s^{j}(X_{i_1},\dots, X_{i_j}; \exists i)\right]^2 \\
		& =\frac{1}{n}\frac{(n-k)^2}{n^2} \sum_{i=1}^n \left[ -\frac{1}{n-1}\sum_{j\neq i}^n s^{1}(X_i) + 
		s^{1}(X_i)  +   \sum_{j=2}^k \right. \\  
		&  \left.- {k-1\choose j-1}{n-1\choose j}^{-1} \sum s^{j}(X_{i_1},\dots, X_{i_j};\not \exists i) 
		+ {k-1\choose j-1}{n-1\choose j-1}^{-1}\sum s^{(j)}(X_{i_1},\dots, X_{i_j}; \exists i)\right]^2 \\
		& = \frac{1}{n}\frac{(n-k)^2}{n^2} \sum_{i=1}^n  \left[s^{1}(X_i) + \mathrm{T}_i \right]^2.
	\end{aligned} 
\end{equation}
$s^{1}(X_i) $ and $\mathrm{T}_i$ are uncorrelated and have mean 0.  After some calculation, we find that 
\begin{equation*}
	\E[(s^{1}(X_i))^2] = V_1
\end{equation*}
and 
\begin{equation*}
	\begin{aligned}
		\E[\mathrm{T}_i^2] &=  \frac{1}{n-1}V_1 + \sum_{j=2}^k  {k-1\choose j-1}^2 \left[{n-1\choose j}^{-1}+ {n-1\choose j-1}^{-1} \right] V_j\\
		& = \frac{1}{n-1}V_1 + \frac{n}{k^2}\sum_{j=2}^k \frac{j}{1-j/n}\frac{{k\choose j}^2}{{n\choose j}}V_j.
	\end{aligned}
\end{equation*}
Then 
\begin{equation}
	\begin{aligned}
		\E[\mathrm{T}_i^2]  &\approx  \frac{1}{n-1}V_1 +  \sum_{j=2}^k {k-1\choose j-1}^2 {n-1\choose j-1}^{-1}V_j\\
		& =   \frac{1}{n-1}V_1 + \sum_{j=2}^k\frac{j}{k}{k-1\choose j-1}{n-1\choose j-1}^{-1}\left[{k\choose j}V_j\right] \\
		& \leq   \frac{1}{n-1}V_1  + \frac{2}{n}\sum_{j=2}^k {k\choose j}V_j \\
		& =  \frac{1}{n-1}\zeta_1  + \frac{2}{n}(\zeta_k-k\zeta_1),
	\end{aligned}
\end{equation}
where $\zeta_k=\V(s)= \sum_{j=1}^k {k\choose j} V_j$ and $\zeta_1=\V(\E[s|X_1])=V_1$.
Let $\rl = \E[(s^1(X_i))^2]$ and $\mathrm{R} = \E[\mathrm{T}_i^2]$. Since $\frac{k}{n}(\frac{\zeta_k}{k\zeta_1}-1)\to 0$, we have 
\begin{equation}
	\label{lr}
	\mathrm{R}/\mathrm{L}\leq \frac{2/n(\zeta_k-k\zeta_1)}{\zeta_1} +\frac{1}{n-1}\to 0.
\end{equation}
Therefore, $(s^1(X_i))^2 $ dominates $\mathrm{T}^2_i$ and thus by \cref{gold},
\begin{equation}
	\label{conv}
	\begin{aligned}
		\frac{1}{n}\sum [e_i-s_0]^2/V_1   & \xrightarrow[]{p} \frac{1}{n} \frac{(n-k)^2}{n^2} \sum_{i=1}^n[ s^{(1)}(X_i)]^2/V_1 \\
		& \xrightarrow[]{p}  \frac{(n-k)^2}{n^2}  \E[ s^{(1)}(X_i)]^2/V_1\\
		& \to 1.
	\end{aligned}
\end{equation}
So $\sij_\ru/\frac{k^2}{n}V_1 = \frac{1}{n}\sum [e_i-s_0]^2/V_1 \xrightarrow[]{p}1$. Observe that 
\begin{equation}
	\begin{aligned}
		1\leq \V(\ru_{n,k})/\frac{k^2}{n}V_1& =\left(\frac{k^2}{n}V_1\right)^{-1} \sum_{j=1}^k{k\choose j}^2{n\choose j}^{-1}V_j  \\
		& \leq 1 + \left(\frac{k^2}{n}V_1\right)^{-1}\frac{k^2}{n^2}\sum_{j=2}^k {k\choose j}V_j \\
		& \leq 1  + \frac{k}{n}(\frac{\zeta_k}{k\zeta_1}-1)\\
		& \to 1.
	\end{aligned}
\end{equation}
Therefore, $\V(\ru_{n,k})/\frac{k^2}{n}V_1\to 1$ and thus $\sij_\ru/ \V(\ru_{n,k})  \xrightarrow[]{p}1$. 

For $s=s(x_1,\dots, x_k;\omega)$, we define an extended H-decomposition by letting
\begin{equation}
	s^1(x_1) = s_1(x_1),
\end{equation}
\begin{equation}
	s^c(x_1,\dots, x_c) = s_c(x_1,\dots, x_c) -\sum_{j=1}^c \sum_{i_1,\dots, i_j\in \{1,\dots, j\}}s^j(x_{i_1},\dots, x_{i_j})
\end{equation}
for $c=1,\dots, k-1$ and 
\begin{equation}
	s^k(x_1,\dots, x_k) = s(x_1,\dots, x_k;\omega) - \sum_{j=1}^{k-1}\sum_{i_1,\dots, i_j \in \{1,\dots, j\}}s^j(x_{i_1},\dots, x_{i_j})
\end{equation}
where $s_c(x)=\E[s(x_1,\dots,x_c,X_{c+1}, X_k;\omega)]-\E[s]$. Then $s(x_1,\dots, x_k;\omega)$ can be written as $\E[s]+ \sum_{j=1}^k\sum_{i_1,\dots, i_j\in \{1,\dots, j\}} s^j(x_{i_1},\dots, x_{i_j})$ and thus for 
\begin{equation}
	\sij^{\omega}_\ru = \frac{k^2}{n^2}\sum [e_i^\omega - s_0^\omega]^2,
\end{equation}
it can be decomposed the same way as  \cref{decom}. Thus, we have $\sij^\omega_\ru /\frac{k^2}{n}\zeta_{1,\omega}\xrightarrow[]{p}1 $. \hfill $\blacksquare$ \\

\noindent \textbf{Proof of \cref{U-2}:} 
As above, let us first ignore the extra randomness $\omega$ for simplicity. 
Letting $p = N{n\choose k}^{-1}$, we have 
\begin{equation}
	\begin{aligned}
		\hat{e}_i - \hat{s}_0 & = \frac{n}{Nk}\sum s(X_{i_1},\dots, X_{i_k};\exists i) - \frac{1}{N}\sum s(X_{i_1},\dots, X_{i_k}) \\
		& = \frac{n}{Nk}\sum (s(X_{i_1},\dots, X_{i_k};\exists i) -\theta) - \frac{1}{N}\sum (s(X_{i_1},\dots, X_{i_k}) -\theta) + \left(\frac{\hat{N}_i}{N_i} - \frac{\hat{N}}{N}\right) \theta \\
		& = {n-1\choose k-1}^{-1}\sum \frac{\rho}{p}(s(X_{i_1},\dots, X_{i_k};\exists i)-\theta) - {n\choose k}^{-1}\sum \frac{\rho}{p}(s(X_{i_1},\dots, X_{i_k}) -\theta)+ r_i \\
		& \triangleq {e}^\dagger_i- {s}^\dagger_0 + r_i.
	\end{aligned}
\end{equation}
where $N_i=Nk/n$, $\hat{N}=\sum \rho$ and $\hat{N}_i = \sum \rho\bone_{i\in \{i_1,\dots, i_k\}}$.

Comparing the H-decomposition of $s^\dagger(x_1,\dots,x_k;\rho) = \frac{\rho}{p}s(x_1,\dots,x_k)$ and $s(x_1,\dots, x_k)$,  we have $V_j^\dagger = V_j$ for $j=1,\dots, k-1$ and $V_k^\dagger = V_k + \frac{1-p}{p}\zeta_k$. Similar to \cref{decom}, we have 
\begin{equation}
	\begin{aligned}
		\frac{1}{n}\sum_{i=1}^n [e^\dagger_i - s^\dagger_0]^2  & = \frac{1}{n}\frac{(n-k)^2}{n^2}\sum_{i=1}^n \left[s^{1}(X_i) +\mathrm{T}^{\dagger}_i\right]^2,
	\end{aligned}
\end{equation}
where $s^{1}(x) = \E[s(x,X_2,\dots, X_k)]$. Note that 
$\E[(s^{1}(X_1))^2]  = V_1^\dagger = V_1 $ and 
\begin{equation}
	\label{lrm}
	\begin{aligned}	
		\E[(\mathrm{T}_i^\dagger)^2] & = \frac{1}{n-1}V^\dagger_1 +  \frac{n}{k^2}\sum_{j=2}^{k}\frac{j}{1-j/n} \frac{{k\choose j}^2}{{n\choose j}} V_j^\dagger \\
		& = \frac{1}{n-1}V_1 +  \frac{n}{k^2}\sum_{j=2}^{k}\frac{j}{1-j/n} \frac{{k\choose j}^2}{{n\choose j}} V_j +\frac{n}{k^2} \frac{k}{1-k/n}\frac{1}{N}(1-p)\zeta_k\\
		& :=  \mathrm{R} + \mathrm{M}
	\end{aligned}
\end{equation}
where $\mathrm{M} =  \frac{1}{1-k/n}\frac{n}{Nk}(1-p)\zeta_k$. Let $\rl = \E[(s^1(X_i))^2]$. Since $\frac{k}{n}(\frac{\zeta_k}{k\zeta_1}-1)\to 0$,  we have $\mathrm{R}/\rl\to 0$ by \cref{lr}. 
Next, we have 
\begin{equation}
	\label{sij-con}
	\begin{aligned}
		\mathrm{M}/\mathrm{L} &=\frac{\frac{1}{1-k/n}\frac{n}{Nk}(1-p)\zeta_k}{\zeta_1}  \\
		& \leq  \frac{1}{1-k/n} \times \frac{n}{N}\frac{\zeta_k}{k\zeta_1}  \\
		& \approx \frac{n}{N}\frac{\zeta_k}{k\zeta_1} \to 0
	\end{aligned}
\end{equation} 
since $\zeta_k$ is bounded and $\frac{n}{Nk\zeta_1}\to 0$.
Thus, $\frac{1}{n}\sum_{i=1}^n [e^\dagger_i - s^\dagger_0]^2/V_1  \xrightarrow[]{p} 0$ by \cref{gold}. Note that
\begin{equation}
	\begin{aligned}
		\E[r_i^2] 
		& =  \E\left[(\frac{\hat{N}_i}{N_i} -1) -(\frac{\hat{N}}{N}-1)\right]^2\theta^2 \\
		& \leq 2\theta^2 \left[\frac{1}{N_i}(1-\frac{N}{{n\choose k}}) + \frac{1}{N}(1-\frac{N}{{n\choose k}})\right] \\
		& \leq 4\theta^2 /N_i.
	\end{aligned}
\end{equation}
Thus, $\frac{1}{n}\sum \E[\sum r_i^2]/V_1 \leq 4 \theta^2  \frac{n}{N} \frac{1}{kV_1}\to 0 $ according to the conditions.  By \cref{gold} again, we have 
$\frac{1}{n}\sum_i (\hat{e}_i -\hat{s}_0)^2/V_1 \xrightarrow[]{p} 1$
and therefore it follows that $\widehat{\sij_\ru}/\frac{k^2}{n}V_1\xrightarrow[]{p}1$.

Again, the extra randomness only results in an extended version of H-decomposition. Everything above can be directly applied to $s(x_1,\dots, x_k;\omega)$. \hfill $\blacksquare$ \\

\section{Higher Order Pseudo Infinitesimal Jackknife }
\label{app:higherorderIJ}

Recall that in the context of U-statistics, $\V(\ru) =\sum_{j=1}^k {k\choose j}^2{n\choose j}^{-1}V_j$.  In the final discussion provided in the main text, we noted that the preceding results largely assumed that the U-statistic was close to linear statistic so that the variance of U-statistic is dominated by its first order term $k^2/n V_1$ and so the problem of providing a good estimate for $\V(\ru)$ can be reduced to providing a good estimate for $V_1$.  But what if the statistic is not close to linear and the remaining terms in $\V(\ru)$ are not negligible?  Can we obtain an improved estimator by proposing further estimates of $V_j$ for $j=2,\dots, k$?  We now address these questions.

We begin by considering the second term $V_2$ and extend those results to all $j$, $3\leq j\leq k$.  Since $V_2=\V(\E[s|X_1,X_2]-\E[s|X_1]-\E[s|X_2]+\E[s])$, a natural estimate for the second order term ${k\choose 2}^2{n\choose 2}^{-1}V_2$ would be 
\begin{equation}
	\label{l-u-2}
	\left( {{k\choose 2}}/{{n\choose 2}} \right)^2 \sum_{i,j} [e_{ij}-e_i-e_j+s_0]^2
\end{equation}
where $e_{ij}=\E_*[s^*|X_1^*=X_i, X_2^*=X_j]$. Before analyzing the properties of this estimate, we first point out its connection to the $\sij_\ru$.

\begin{proposition}
	\label{sij-d}
	Let $\D_n^*= (X_1^*,\dots, X_k^*)$ be a subsample of $\D_n$ and $w^*_{ij} = \bone_{X_i,X_j\in \D_n^*} - \frac{k}{n}\bone_{X_i\in \D_n^*} - \frac{k}{n}\bone_{X_j\in \D_n^*} +\frac{k(k-1)}{n(n-1)}$.  Then 
	\begin{equation}
		\cov_*(s^*, w_{ij}^*) = {k\choose 2}/{n\choose 2 } (e_{ij}-e_i-e_j+s_0)
	\end{equation} 
	where $*$ refers the procedure of subsampling without replacement  and $e_{ij} = \E_*[s^*|X_1^*=X_i,X_2^*=X_j]$. We  call \cref{l-u-2} the  second order pseudo-IJ estimator of U-statistics:
	\begin{equation}
		\begin{aligned}
			\sij_\ru(2) &= \sum_{i,j}\cov^2_*(s^*,w_{ij}^*) \\
			& = {{k\choose 2}^2}/{n\choose 2}^2 \sum_{i,j} [e_{ij}-e_i-e_j+s_0]^2.
		\end{aligned}
	\end{equation}
\end{proposition}

\begin{proof}
	We have 
	\begin{equation}
		\begin{aligned}
			~~~~ & (e_{12}-e_1-e_2+s_0) \\
			&  = \sum_{w_1^*=1,w_2^*=1,w_1^*+\dots + w_n^*=k}\frac{(k-2)!}{(n-2)\cdots (n-k)}s^* - \sum_{w_1^*=1,w_1^*+\dots + w_n^*=k} \frac{(k-1)!}{(n-1)\cdots (n-k)} s^* - \\
			& ~~~~ \sum_{w_2^*=1,w_1^*+\dots + w_n^*=k} \frac{(k-1)!}{(n-2)\cdots (n-k)}s^* 
			+ \sum_{w_1^*+\dots + w_n^*=k} \frac{k!}{n\cdots (n-k)} s^* \\
			&  = \sum_{w_1^*+\dots + w_n^*=k}\frac{n(n-1)}{k(k-1)}{n\choose k}^{-1}(w_1^*)(w_2^*)s^* - \sum_{w_1^*+\dots + w_n^*=k} \frac{n}{k}{n\choose k}^{-1}(w_1^*) s^* -\\
			& ~~~~ \sum_{w_1^*+\dots + w_n^*=k}  \frac{n}{k}{n\choose k}^{-1}(w_2^*)s^* + \sum_{w_1^*+\dots + w_n^*=k} {n\choose k}^{-1}s^*  \\
			& = \sum {n\choose k}^{-1}\left(\frac{n(n-1)}{k(k-1)} w_1^*w_2^* - \frac{n}{k}w_1^* - \frac{n}{k}w_2^*+1\right)s^* \\
			& = \frac{n(n-1)}{k(k-1)}  \sum {n\choose k}^{-1}\left(w_1^*w_2^* - \frac{k-1}{n-1}w_1^* - \frac{k-1}{n-1}w_2^*+\frac{k(k-1)}{n(n-1)}\right)s^*. \\
		\end{aligned}
	\end{equation}
	Thus, 
	\begin{equation}
		\sum_{i,j} \cov_*^2(s^*, w_{ij}^*) = \left(\frac{{k\choose  2}}{{n\choose 2}}\right)^2 \sum_{i,j} (e_{1,2}-e_1-e_2+s_0)^2.
	\end{equation}
\end{proof}

Note that the first-order $\sij_\ru$ involves the covariance of $s^*$ and $w_j^*$ -- the counts of how many times each original observation appears in a subsample, whereas $\sij_\ru(2)$ involves covariance the of $s^*$ and $w_{ij}^*$ -- the count of how often each \emph{pair} of observations appears in a subsample.  In this sense then, $\sij_\ru(2)$ is a natural extension of $\sij_\ru$ and for notational convenience, we can also write $\sij_\ru$ as $\sij_\ru(1)$.  Similarly, we can extend this idea to derive a general $d^{th}$ order estimator $\sij_\ru(d)$ for $d=1,\dots, k$.

\begin{corollary}
	For $d=1,\dots, k$, define
	\begin{equation}
		\begin{aligned}
			\sij_\ru(d)  &= \sum_{(n,d)}\cov_*^2(s^*, w^*_{i_1,\dots, i_d}) \\
			& = {k\choose d}^2/{n\choose d}^2 \sum_{(n,d)}\left[\sum_{j=0}^d (-1)^{d-j}\sum_{(d,j)}e_{i_1,\dots, i_j}\right]^2
		\end{aligned}
	\end{equation}
	where $w_{i_1,\dots, i_d}^* = \sum_{j=0}^d  (-1)^{d-j}\frac{{n-d+j\choose k-d+j}}{{n\choose k}} \left[\sum_{(d,j)}\prod w^*_{i_j}\right]$. The expression for $w^*_{i_1,\dots, i_d}$ is somewhat involved because we are considering subsampling without replacement. If instead we perform subsampling with replacement, then  $w_{i_1,\dots, i_d}^* = \prod(w^*_{i_j}-1)$.
\end{corollary}

Like $\E[\sij_\ru]$, $\E[\sij_\ru(d)]$ is a linear combination of the $V_j$. Let $a_i = {n -i \choose k-i}^{-1}$ for $ i = 0,1,\dots, d$ and define $b_i$ for $i=0,1,\dots, d$ by 
\[
\begin{aligned}
	b_0 &=   a_0 \\
	b_1  & = a_1-a_0  =  a_1 - b_0 \\
	\vdots \\
	b_d & = a_d-{d\choose 1}a_{d-1} + {d\choose 2}a_{d-2} - \dots a_0   =  a_d - {d\choose 1}b_{d-1}  -{d\choose 2}b_{d-2}- \dots   -b_0.
\end{aligned}
\]
Additionally, let $c_i = b_i {n-d\choose k-i}$ and $m_i = c_{d-i}$ for $i = 0,\dots, d$. Then for $j=1,\dots, k$, the coefficient of $V_j$ in $\E[\sij_\ru(d)]$ is ${k\choose d}^2/{n\choose d} \lambda_j(d)$, where 
\[
\begin{aligned}
	\lambda_j(d) 
	& = {d\choose 0}{n-d\choose j-d}^{-1} \left(m_0 {n-d\choose j-d}\right)^2 +\\
	& ~~~~ {d\choose 1}{n-d\choose j-d+1)}^{-1}\left[m_1{k-d+1 \choose j-d+1}- m_0 {k-d \choose j-d+1}\right]^2+ \\
	& ~~~~ {d\choose 2}{n-d\choose j-d+2}^{-1}\left[m_2{k-d+2\choose j-d+2} -{2\choose 1}m_1{k-d+1\choose j-d+2}  + m_0{k-d\choose j-d+2)}\right]^2+ \\
	& ~~~~ \vdots \\
	& ~~~~ {d\choose d}{n-d\choose j}^{-1}\left[m_d{k\choose j}-{d\choose d-1}m_{d-1}{k-1\choose j} \right.\\
	& ~~\left.~~~~+ {d\choose d-2}m_{d-2}{k-2 \choose j}-\dots {d\choose 1}m_0{k-d\choose j}\right]^2.
\end{aligned}
\]
Putting this all together, we have the follow result.
\begin{proposition}
	\label{sij-bias}
	Writing the $\V(\ru)$ and $\E[\sij_\ru(d)]$ in terms of $V_1,\dots, V_k$,  the ratio of the coefficient of $V_j$ in  $\E[\sij_\ru(d)]$ to that in $\V(\ru)$ is given by $r_j(d)$, where 
	\begin{equation}
		\begin{aligned}
			r_j(d) = \frac{\lambda_j(d) {k\choose d}^2 {n\choose d}^{-1}}{{k\choose j}^2{n\choose j}^{-1}}, \quad j = 1,\dots, k.
		\end{aligned}
	\end{equation}
	Furthermore, $r_j(d)$ is monotone increasing with respect to $j$.
\end{proposition}

\begin{proof}	
	We derive $\E[\sij_\ru(2)]$ in work that follows.  Expressions for $d\geq 3$ can be derived in the same spirit.  
	Consider $e_{ij} = \E_*[s^*|X_1^*=X_1, X_2^* = X_2]$.  We have
	\begin{equation*}
		\begin{aligned}
			(e_{12}- e_1-e_2+s_0) 
			& = {n\choose k}^{-1}\sum s(X_{i_1},\dots, X_{i_k};\not \exists 1,\not \exists 2) + \\
			& ~~~~ \left({n-2\choose k-2}^{-1}-2{n-1\choose k-1}^{-1} + {n\choose k}^{-1}\right)\sum s(X_{i_1}, \dots, X_{i_k};\exists 1,\exists 2)- \\
			&~~~~  \left({n-1\choose {k-1}}-{n\choose k}^{-1}\right)\sum s(X_{i_1},\dots, X_{i_k};\exists 1, \not \exists 2) - \\
			&  ~~~~ \left({n-1\choose {k-1}}-{n\choose k}^{-1}\right)\sum  s(X_{i_1},\dots, X_{i_k};\not \exists 1 \exists 2)  \\
			& := \mathrm{I} +  \mathrm{II}  + \mathrm{III} +\mathrm{IV}. 
		\end{aligned}
	\end{equation*}
	Looking at each term individually, by H-decomposition we have 
	\begin{equation}
		\begin{aligned}
			\mathrm{I} & = {n\choose k}^{-1}{n-2\choose k}{n-2\choose k}^{-1}\sum s(X_{i_1}, \dots, X_{i_k}; \not \exists 1, \not \exists 2) \\
			& = {n\choose k}^{-1}{n-2\choose k}\sum_{j=1}^k {k\choose j}{n-2\choose j}^{-1}\sum s^{j }(X_{i_1},\dots, X_{i_j}; \not \exists 1,\not \exists 2)
		\end{aligned}
	\end{equation}
	
	\begin{equation*}
		\begin{aligned}
			\mathrm{II}& =  - \left({n-1\choose {k-1}}^{-1}-{n\choose k}^{-1}\right)\sum s(X_{i_1},\dots, X_{i_k};\exists 1, \not \exists 2)  \\
			& =  - \left({n-1\choose {k-1}}^{-1}-{n\choose k}^{-1}\right){n-2\choose k-1}{n-2\choose k-1}^{-1}\sum s(X_{i_1},\dots, X_{i_k}; \exists 1, \not \exists 2) \\
			&  =  - \left({n-1\choose {k-1}}^{-1}-{n\choose k}^{-1}\right){n-2\choose k-1}\left[\sum_{j=1}^{k-1} {k-1\choose j}{n-2\choose j}^{-1}\sum s^{j}(X_{i_1},\dots, X_{i_j}; \not \exists 1,\not \exists 2) \right.+\\
			& ~~~~  \left. \sum_{j=1}^{k}{k-1\choose j-1}{n-2\choose j-1}^{-1}\sum s^{j}(X_{i_1},\dots, X_{i_j}; \exists 1, \not \exists 2)\right]
		\end{aligned}
	\end{equation*}
	
	\begin{equation*}
		\begin{aligned}
			\mathrm{III}& =  - \left({n-1\choose {k-1}}^{-1}-{n\choose k}^{-1}\right)\sum s(X_{i_1},\dots, X_{i_k};\not \exists 1,  \exists 2) \\
			& =  - \left({n-1\choose {k-1}}^{-1}-{n\choose k}^{-1}\right){n-2\choose k-1}{n-2\choose k-1}^{-1}\sum s(X_{i_1},\dots, X_{i_k};\not \exists 1,  \exists 2) \\
			& =    - \left({n-1\choose {k-1}}^{-1}-{n\choose k}^{-1}\right){n-2\choose k-1}\left[\sum_{j=1}^{k-1} {k-1\choose j}{n-2\choose j}^{-1}\sum s^{j}(X_{i_1},\dots, X_{i_j};\not \exists 1,\not \exists 2) \right.+\\
			& ~~~~  \left. \sum_{j=1}^{k}{k-1\choose j-1}{n-2\choose j-1}^{-1}\sum s^{j}(X_{i_1},\dots, X_{i_j}; \not \exists 1,  \exists 2)\right]
		\end{aligned}
	\end{equation*}
	
	\begin{equation}
		\begin{aligned}
			\mathrm{IV}& =  \left({n-2\choose k-2}^{-1}-2{n-1\choose k-1}^{-1} + {n\choose k}^{-1}\right)\sum s(X_{i_1},\dots, X_{i_k};\exists 1, \exists 2)  \\
			& = \left({n-2\choose k-2}^{-1}-2{n-1\choose k-1}^{-1} + {n\choose k}^{-1}\right){n-2\choose  k-2} \left[\sum_{j=1}^{k-2} {k-2\choose j}{n-2\choose j}^{-1} \right.\times\\
			&~~~~~~~~ \sum s^j(X_{i_1},\dots, X_{i_j}; \not \exists 1, \not \exists 2)+ \\
			& ~~~~~~~~  \left. \sum_{j=1}^{k-1}{k-2\choose j-1}{n-2\choose j-1}^{-1}\times \sum s^j(X_{i_1},\dots, X_{i_j}; \exists 1, \not \exists 2) \right.+\\
			& ~~~~~~~~  \left.  \sum_{j=1}^{k-1}{k-2\choose j-1}{n-2\choose j-1}^{-1}\times \sum s^j(X_{i_1},\dots, X_{i_j}; \not  \exists 1,  \exists 2 )\right.+ \\
			& ~~~~~~~~  \left. \sum_{j=2}^{k}{k-2\choose j-2}{n-2\choose j-2}^{-1}\times \sum s^j(X_{i_1},\dots, X_{i_j};\exists 1, \exists 2) \right] \\
		\end{aligned}
	\end{equation}
	
	\noindent In conclusion, we have $\mathrm{I+II+III+IV} = A +B+C+D$, where $A,B,C,D$ are uncorrelated and given by
	
	\begin{equation*}
		\begin{aligned}
			A 
			& = \left({n-2\choose k-2}^{-1}-2{n-1\choose k-1}^{-1} + {n\choose k}^{-1}\right){n-2\choose  k-2}\times\\ &~~~~\sum_{j=2}^{k}{k-2\choose j-2}{n-2\choose j-2}^{-1}\sum s^j(X_{i_1},\dots, X_{i_j};\exists 1, \exists 2) 
		\end{aligned}
	\end{equation*}

	\begin{equation*}
		\begin{aligned}
			B & = - \left({n-1\choose {k-1}}-{n\choose k}^{-1}\right){n-2\choose k-1}\left[ \sum_{j=1}^{k}{k-1\choose j-1}{n-2\choose j-1}^{-1}\sum s^j(X_{i_1},\dots, X_{i_j}; \exists 1, \not \exists 2) \right]+ \\
			&~~~~   \left({n-2\choose k-2}^{-1}-2{n-1\choose k-1}^{-1} + {n\choose k}^{-1}\right){n-2\choose  k-2}\times \\
			& ~~~~~~~~\left. \sum_{j=1}^{k-1}{k-2\choose j-1}{n-2\choose j-1}^{-1}\sum s^j(X_{i_1},\dots, X_{i_j};   \exists 1, \not  \exists 2 ) \right.
		\end{aligned}
	\end{equation*}
	\begin{equation*}
		\begin{aligned}
			C & = - \left({n-1\choose {k-1}}-{n\choose k}^{-1}\right){n-2\choose k-1}\left[ \sum_{j=1}^{k}{k-1\choose j-1}{n-2\choose j-1}^{-1}\sum s^j(X_{i_1},\dots, X_{i_j}; \not \exists 1,  \exists 2) \right] +\\
			&~~~~   \left({n-2\choose k-2}^{-1}-2{n-1\choose k-1}^{-1} + {n\choose k}^{-1}\right){n-2\choose  k-2} \times \\ & ~~~~~~~~\left.\sum_{j=1}^{k-1}{k-2\choose j-1}{n-2\choose j-1}^{-1}\sum s^j(X_{i_1},\dots, X_{i_j};  \not \exists 1,   \exists 2 ) \right.
		\end{aligned}
	\end{equation*}
	
	\begin{equation*}
		\begin{aligned}
			D 
			& = {n\choose k}^{-1}{n-2\choose k}\sum_{j=1}^k {k\choose j}{n-2\choose j}^{-1}\sum s^j(X_{i_1},\dots, X_{i_j}; \not \exists 1,\not \exists 2)- \\
			& ~~~~  \left({n-1\choose {k-1}}-{n\choose k}^{-1}\right){n-2\choose k-1}\left[\sum_{j=1}^{k-1} {k-1\choose j}{n-2\choose j}^{-1}\sum s(X_{i_1},\dots, X_{i_j};\not \exists 1,\not \exists 2) \right] -\\
			& ~~~~  \left({n-1\choose {k-1}}-{n\choose k}^{-1}\right){n-2\choose k-1}\left[\sum_{j=1}^{k-1} {k-1\choose j}{n-2\choose j}^{-1}\sum s^j(X_{i_1},\dots, X_{i_j};\not \exists 1,\not \exists 2) \right]+\\
			& ~~~~  \left({n-2\choose k-2}^{-1}-2{n-1\choose k-1}^{-1} + {n\choose k}^{-1}\right){n-2\choose  k-2}\times \\  
			& ~~~~~~~~ \sum_{j=1}^{k-2} {k-2\choose j}{n-2\choose j}^{-1}\sum s^j(X_{i_1},\dots, X_{i_j}; \not \exists 1, \not \exists 2).
		\end{aligned}
	\end{equation*}
	
	\noindent 
	Let $C_2= \left( {n-2\choose k-2}^{-1}-2{n-1\choose k-1}^{-1} + {n\choose k}^{-1}\right) {n-2\choose  k-2}$, $C_1= \left({n-1\choose {k-1}}^{-1}-{n\choose k}^{-1}\right){n-2\choose k-1} $ and $C_0= {n\choose k}^{-1}{n-2\choose k} $.  Then we have
	\begin{equation}
		\begin{aligned}
			\V(A) & =\sum_{j=2}^{k} {n-2\choose j-2}^{-1}\left(C_2 {k-2\choose j-2}\right)^2V_j\\
			& =\sum_{j=2}^{k} {n-2\choose j-2}^{-1}\left(C_2 {k-2\choose j-2}\right)^2V_j
		\end{aligned}
	\end{equation}

	\begin{equation}
		\begin{aligned}
			\V(B) &= \sum_{j=1}^{k-1} {n-2\choose j-1}\left(- C_1{k-1\choose j-1}{n-2\choose j-1}^{-1} + C_2 {k-2\choose j-1}{n-2\choose j-1}^{-1}\right)^2V_j +\\
			& ~~~~ {n-2\choose k-1}^{-1}\left(-C_1{k-1\choose k-1}\right)^2V_k\\
			& = \sum_{j=1}^k {n-2\choose j-1}^{-1}   \left(- C_1{k-1\choose j-1}+C_2 {k-2\choose j-1}\right)^2V_j \\
		\end{aligned}
	\end{equation}

	\begin{equation}
		\begin{aligned}
			\V(B) & = \V(C)
		\end{aligned}
	\end{equation}

	\begin{equation}
		\begin{aligned}
			\V(D) & = \sum_{j=1}^{k-2} {n-2\choose j}^{-1}\left(C_0{k\choose j} -2C_1{k-1\choose j} + C_2{k-2\choose j}\right) + \\
			& ~~~~  {n-2\choose k-1}^{-1} \left(C_0{k\choose k-1} -2C_1{k-1\choose k-1}\right)+ \\
			& ~~~~ {n-2\choose k}C_0{k\choose k}V_k \\
			& = \sum_{j=1}^{k} {n-2\choose j}^{-1}\left(C_0{k\choose j} -2C_1{k-1\choose j} + C_2{k-2\choose j}\right)V_j.
		\end{aligned}
	\end{equation}

	\noindent Therefore $\E[\sij_\ru(2)] = {k\choose  2}^2/{n\choose 2}\sum_{j=1}^k \lambda_j(2)V_j$, where
	\begin{equation}
		\begin{aligned}
			\lambda_j(2)& =  {n-2\choose j-2}^{-1}\left(C_2 {k-2\choose j-2}\right)^2+ \\
			& ~~~~  2{n-2\choose j-1}^{-1}   \left(- C_1{k-1\choose j-1}+C_2 {k-2\choose j-1}\right)^2+\\
			& ~~~~   {n-2\choose j}^{-1}\left(C_0{k\choose j} -2C_1{k-1\choose j} + C_2{k-2\choose j}\right)^2.  \\
		\end{aligned}
	\end{equation} 
\end{proof} 
\begin{figure}[t]
	\centering
	\includegraphics[width=\textwidth,height=\textheight,keepaspectratio]{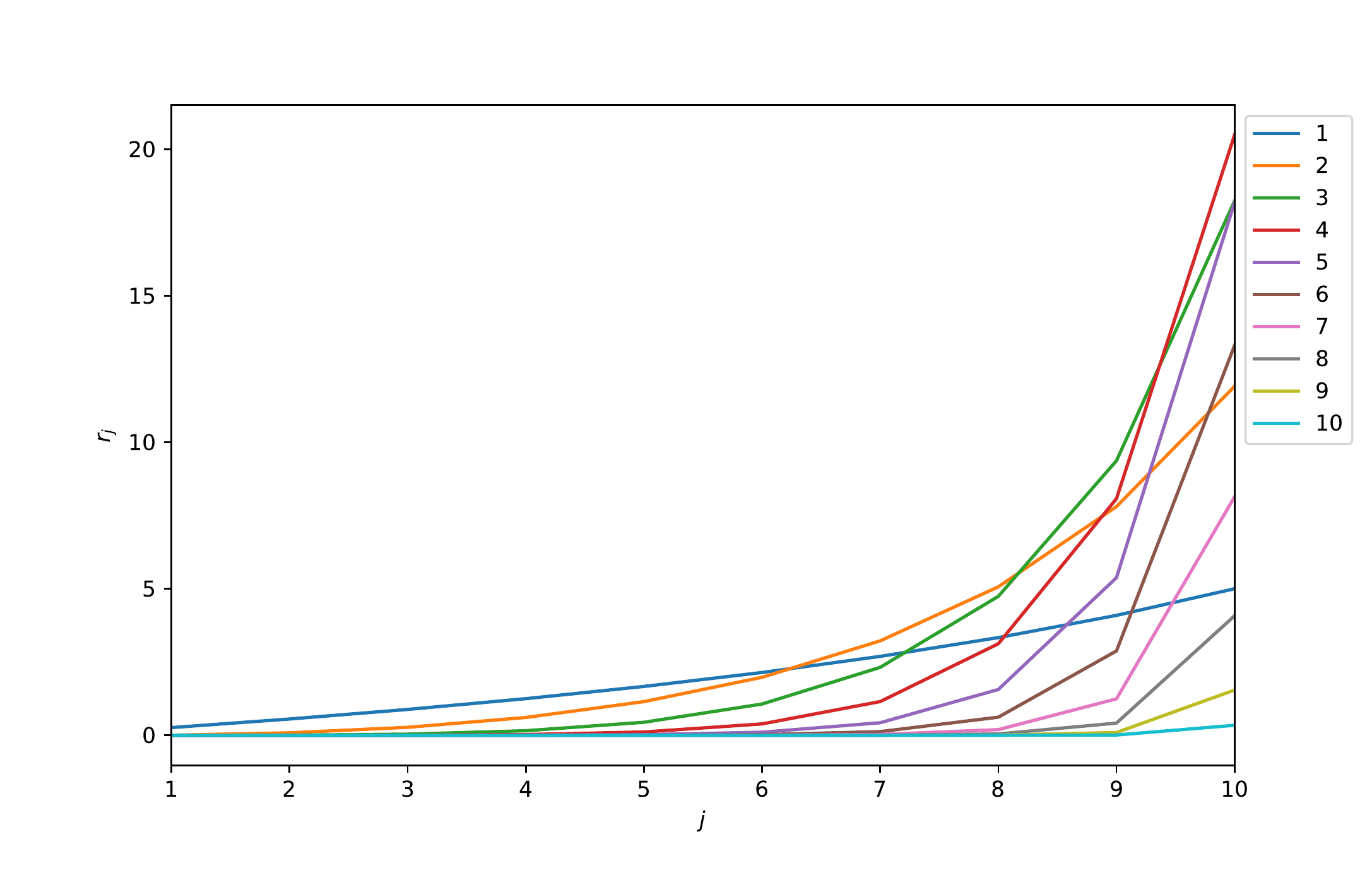}
	\caption{A plot of $\{r_j(d)\}_{j=1}^k$,  where n=20 and k=10}
	\label{rjd}
\end{figure}  
As a simple example, we can take $n=20$ and $k=10$ and plot the curve of $r_j(d)$ to get some insight into how it behaves.  In the interest of consistency, we want for $r_j/r_1$ to be close to 1, at least for small $j$, because $\V(\ru)$ should be dominated by the first several terms. From Figure \ref{rjd} , it appears that $\sij_\ru(1)$ still perform better than the other higher-order estimates.  It is possible that combining $\sij_\ru(d)$ for $d=1,\dots, k$ in some way could yield an estimator that outperforms $\sij_\ru(1)$; this is a potentially interesting topic for future research.

\end{document}